\documentclass[11pt,letterpaper]{amsart}
\usepackage[top=1.15in,bottom=1.15in,left=1.15in,right=1.15in,marginpar=1in]{geometry}
\usepackage{graphicx}
\usepackage{amssymb}
\usepackage{epstopdf}
\usepackage[mathcal]{eucal}
\usepackage{hyperref}
\usepackage{enumitem}

\usepackage{tikz-cd}
\usepackage{diagbox}
\usepackage{tkz-euclide} 
\usepackage[all,cmtip]{xy}

\usepackage[textsize=tiny]{todonotes}

\usepackage{scalerel}

\makeatletter
\newcommand*{\bigcdot}{}
\DeclareRobustCommand*{\bigcdot}{%
  \mathbin{\mathpalette\bigcdot@{}}%
}
\newcommand*{\bigcdot@scalefactor}{.75}
\newcommand*{\bigcdot@widthfactor}{1.15}
\newcommand*{\bigcdot@}[2]{%
  \sbox0{$#1\vcenter{}$}
  \sbox2{$#1\cdot\m@th$}%
  \hbox to \bigcdot@widthfactor\wd2{%
    \hfil
    \raise\ht0\hbox{%
      \scalebox{\bigcdot@scalefactor}{%
        \lower\ht0\hbox{$#1\bullet\m@th$}%
      }%
    }%
    \hfil
  }%
}
\makeatother
\newcommand{\nc}{\newcommand}
\nc{\dmo}{\DeclareMathOperator}
\nc{\rnc}{\renewcommand}
\nc{\nt}{\newtheorem}

\nt{theorem}{Theorem}[section]
\nt*{ubertheorem}{Nielsen--Thurston \"Ubertheorem}
\nt{proposition}[theorem]{Proposition}
\nt{corollary}[theorem]{Corollary}
\nt{lemma}[theorem]{Lemma}

\let\ORIincludegraphics\includegraphics
\renewcommand{\includegraphics}[2][]{\ORIincludegraphics[scale=0.91666,#1]{#2}}

\dmo{\Teich}{Teich}
\dmo{\PMod}{PMod}
\dmo{\Crit}{Crit}

\nc{\F}{\mathcal F}
\nc{\rpt}{{\mathbb R}{\mathrm P}^2}
\rnc{\O}{\mathcal O}

\nc{\A}{\mathcal A}
\nc{\Y}{\mathcal Y}
\nc{\T}{\mathcal T}
\nc{\X}{\mathcal X}

\nc{\C}{\mathbb C}
\nc{\Z}{\mathbb Z}
\rnc{\H}{\mathbb H}
\nc{\R}{\mathbb R}

\nc{\HM}{\mathcal {HM}}
\nc{\SHM}{\mathcal {SHM}}

\nc{\p}[1]{\bigskip\noindent\emph{#1.}}

\title[Thurston's theorem and the Nielsen--Thurston classification]{Thurston's theorem and the Nielsen--Thurston classification via Teichm\"uller's theorems}
\author{James Belk}
\author{Dan Margalit}
\author{Rebecca R. Winarski}
\address{James Belk \\ School of Mathematics \& Statistics\\ 15 University Gardens \\ University of Glasgow \\ G12 8QW \\ james.belk@glasgow.ac.uk}
\address{Dan Margalit \\ Department of Mathematics\\ Vanderbilt University \\ 1326 Stevenson Center Ln \\ Nashville, TN 37240 \\ dan.margalit@vanderbilt.edu}
\address{Rebecca R. Winarski \\  Department of Mathematics and Computer Science\\ College of the Holy Cross\\
1 College Street
Worcester, MA 01610 \\ rebecca.winarski@gmail.com}

\thanks{This material is based upon work supported by the National Science Foundation under Grant Nos.\ DMS-1854367, DMS-1928930, DMS-2002951, and DMS-2203431 and the Engineering and Physical Sciences Research Council under Grant No.\ EP/R032866/1.}

\begin{document}

\maketitle

\begin{abstract}
We give a unified and self-contained proof of the Nielsen--Thurston classification theorem from the theory of mapping class groups and Thurston's characterization of rational maps from the theory of complex dynamics (plus various extensions of these).  Our proof follows Bers' proof of the Nielsen--Thurston classification.
\end{abstract}

\section{Introduction}

The main theorem of this paper is what we call the Nielsen--Thurston \"Ubertheorem.  This is a unification, and extension, of the Nielsen--Thurston classification theorem from the theory of mapping class groups and Thurston's characterization of rational maps from the theory complex dynamics.  The unified statement we give here is new, although the content is almost entirely due to Thurston.  We give a unified proof of the \"Ubertheorem by extending the Bers proof of the Nielsen--Thurston classification \cite{bers,primer} to the case of nontrivial (branched) covers, possibly with marked points that are not post-critical.  In Appendix~\ref{sec:ext}, we also extend the theorem to treat the cases of non-orientable surfaces, orientation-reversing maps, and equivariant maps.

Thurston proved his characterization of rational maps in 1982 and gave several lectures on the proof.  The first published proof was given by Douady and Hubbard in 1993~\cite{DH}.  Our proof of the \"Ubertheorem is not only an extension of the Bers proof of the Nielsen--Thurston classification, but it also tracks the Douady--Hubbard paper closely.  One aim of this paper is to clarify the connection between these two proofs, which have long been recognized to be similar in spirit but have not heretofore been put into a single framework.

\bigskip

The Nielsen--Thurston \"Ubertheorem classifies dynamical branched covers, which we presently define.  Let $\Sigma$ be a marked surface, that is, a pair $(S,P)$ where $S$ is a closed surface, and $P$ is a finite set of marked points in $S$.  By a \textit{dynamical branched cover} of $\Sigma$, we mean a branched covering map $f\colon \Sigma \to \Sigma$ where $f(P) \subseteq P$ and $P$ contains all of the critical values of~$f$.  Dynamical branched covers of the sphere with degree at least 2 are traditionally called Thurston maps (according to Douady--Hubbard, this terminology was suggested by Milnor).

A dynamical branched cover can be a homeomorphism, a nontrivial covering map, or a nontrivial branched covering map.  The last two cases only arise when $S$ is $T^2$ or $S^2$, respectively.  A motivation for studying dynamical branched covers is that they make topological operations accessible in the context of rational maps.  For instance the mating of two polynomials of degree $d$ is a dynamical branched cover of $S^2$ (the maps on the hemispheres being given by the two polynomials) with no complex structure attached.

The Nielsen--Thurston \"Ubertheorem classifies dynamical branched covers up to homotopy.  Here, two dynamical branched covers $f$ and $g$ of $\Sigma$ are homotopic if there is a homeomorphism $h$ of $\Sigma$ that is homotopic to the identity (rel $P$) and satisfies $f \circ h = g$ (this relation is finer than the usual notion of Thurston equivalence; see below).  Before stating the \"Ubertheorem, we recall the statements of the Nielsen--Thurston classification and Thurston's characterization of rational maps.

\subsection{The Nielsen--Thurston classification}   The Nielsen--Thurston classification theorem for surface homeomorphisms \cite[Theorem 13.2]{primer} is a theorem of Thurston from 1974, although the first complete,  published proof was given in 1979 by Fathi--Laudenbach--Po\'enaru \cite{flp_original} (many other proofs have appeared since then).  

In the statement we say that a homeomorphism is \emph{periodic} if some nontrivial power is the identity.  Every periodic homeomorphism is geometric in the sense that it is an isometry in some metric of constant curvature.

Next, we say that a homeomorphism is \emph{reducible} if it preserves a multicurve, that is, a collection of pairwise disjoint simple closed curves in $\Sigma$.  

Finally, a surface homeomorphism $f$ of $\Sigma = (S,P)$ is \emph{pseudo-Anosov} if there is a pair of transverse measured foliations $(\F^+,\F^-)$ on $\Sigma$ that is preserved by $f$ and satisfies
\[
f^{-1}(\F^+,\F^-)=(\lambda\, \F^+,\tfrac{1}{\lambda}\,\F^-) 
\]
for some $\lambda > 1$.  The foliations may have 1-pronged singularities and $k$-pronged singularities with $k \geq 3$.  Each 1-pronged singularity must be at a point of~$P$.  As with periodic maps, pseudo-Anosov maps are geometric in that they preserve the affine structure on $\Sigma$ induced by the pair of measured foliations.  

\begin{theorem}[Nielsen--Thurston classification]
Let $f\colon \Sigma \to \Sigma$ be a homeomorphism, where $\Sigma$ is a closed surface with finitely many marked points. Then $f$ is isotopic to a homeomorphism of one of the following types:
\begin{enumerate}
    \item periodic, 
    \item reducible, or
    \item pseudo-Anosov. 
\end{enumerate}
Type (3) is exclusive from the other two.  If $f$ is of type (3) the pseudo-Anosov structure is unique up to isotopy.
\end{theorem}

We can rephrase this classification as: every homeomorphism decomposes along reducing curves into homeomorphisms that are geometric, that is, periodic or pseudo-Anosov.  

Thurston proved the exclusivity by showing that pseudo-Anosov maps increase the length of every simple closed curve exponentially under iteration (see Section~\ref{sec:proof} for more details).  This clearly fails for periodic and reducible maps (in both cases, some power of the map fixes a curve).  So in this sense the Nielsen--Thurston classification says that the only obstructions  to pseudo-Anosovity are the ``obvious'' ones.  

\subsection{Thurston's characterization of rational maps} Our next goal is to state Thurston's characterization of rational maps from the theory complex dynamics.  (Within the field of complex dynamics, this theorem is often referred to as simply ``Thurston's theorem''; we prefer to avoid this terminology due to the ubiquity of Thurston's work in the fields of mapping class groups, complex dynamics, and beyond.)  Our phrasing of the theorem requires the notion of an unmarked map and the notion of a strong reduction system.

\p{Marked and unmarked maps} We say that a dynamical branched cover $f : (S,P) \to (S,P)$ is unmarked if $P$ is the post-critical set for $f$, that is, the set of $f^k(c)$ where $c$ is a critical point for $f$ and $k \geq 1$.  If $P$ strictly contains the post-critical set, then we say that $f$ is marked.  We can define isotopy for dynamical branched covers in the same way that we defined homotopy; these notions are equivalent since homotopic homeomorphisms of a marked closed surface are isotopic.  

\p{Exceptional maps} We now define exceptional maps of the torus and the sphere (exceptional maps of $S^2$ are examples of Latt\`es-type maps; see below).  We focus here on the unmarked exceptional maps, the marked exceptional maps being obtained from the unmarked ones by adding additional marked points (the latter being not post-critical).  While the notion of exceptional maps allows us to give a sharper and more general theorem, the \"Ubertheorem and its proof make sense without the exceptional cases.  

First, an (unmarked) dynamical branched cover of $T^2$ is exceptional if it has degree greater than 1.  All such maps are (unbranched) covering maps.  The exceptional maps of the sphere will be defined in terms of hyperelliptic involutions of $T^2$, which we now discuss.

A \emph{hyperelliptic involution} $\iota : T^2 \to T^2$ is a homeomorphism of order 2 that acts by $-I$ on $H_1(T^2)$.  Every hyperelliptic involution has exactly four fixed points (this follows, for instance, from the Riemann--Hurwitz formula).  One way to obtain a hyperelliptic involution is to choose an affine structure and base point on $T^2$ and take the linear map given by $-I$.   All other hyperelliptic involutions of $T^2$ are topologically conjugate to this one.

Given a hyperelliptic involution $\iota$, we may regard the quotient $T^2/\iota$ as the sphere $S^2$ with a set $P_0$ of four marked points, the images of the fixed points of $\iota$.  Any dynamical branched cover $f: T^2 \to T^2$ that commutes with $\iota$ descends to an unmarked dynamical branched cover $\bar f$ of the quotient $(S^2,P_0)$.  We refer to any such $f$ as \emph{symmetric} (note that $f$ may permute the fixed points of $\iota$).  Any $\bar f$ constructed in this way is what we call an unmarked exceptional dynamical branched cover of $S^2$.  

If we regard $\iota$ as the linear map given by $-I$, then every linear map of $T^2$ is symmetric, and thus descends to an unmarked exceptional dynamical branched cover of $S^2$.  Further, every dynamical branched cover of $T^2$ is homotopic to a linear one, and so every such cover has a corresponding exceptional map of $S^2$.  This correspondence between homotopy classes is not a bijection; for instance the identity map of $T^2$ and translation by $1/2$ in one (or both) factors are homotopic maps of $T^2$, but the corresponding maps of $S^2$ are not homotopic (they act differently on the set of marked points).

\p{Strong reduction systems and stable multicurves} A labeling of a multicurve is a choice of positive real number for each component of the multicurve.  If two components of a multicurve bound an annulus disjoint from $P$, then we may obtain a related multicurve by replacing these components with a single component whose label is the sum of the two labels.  We may also obtain a related multicurve by deleting any inessential components.  We consider labeled multicurves up to the equivalence relation generated by these two relations and homotopy (where homotopies are not allowed to pass through a marked point).  We say that a representative of an equivalence class is \emph{standard} if it has the minimal number of connected components.  

We may say that a labeled multicurve $\Gamma_1$ \emph{contains} a labeled multicurve $\Gamma_2$ if for each component of the standard representative of $\Gamma_2$ there is a component of the standard representative of $\Gamma_1$ that is homotopic and has a label that is at least as large.  

Given a dynamical branched cover $f : \Sigma \to \Sigma$ and a labeled multicurve $\Gamma$ we obtain a labeled multicurve $f^*(\Gamma)$ whose components are the components of $f^{-1}(\Gamma)$ and whose label at a component $\alpha$ is $1/\deg(f|\alpha)$ times the label of $f(\alpha)$.  Finally, we say that a labeled multicurve $\Gamma$ is a strong reduction system for $f$  if the labeled multicurve $f^*(\Gamma)$ contains the labeled multicurve $\Gamma$.  

If $\Gamma$ is an unlabeled multicurve (or the unlabeled multicurve underlying a labeled one) and $f^*(\Gamma)$ contains $\Gamma$ as unlabeled multicurves, then we say that $\Gamma$ is \emph{stable}.  Similarly, if $f^*(\Gamma)$ equals $\Gamma$ as unlabeled multicurves, we say $\Gamma$ is \emph{invariant}.  

\p{Statement of Thurston's characterization of rational maps} We say that a self-map of $S^2$ is \emph{rational} if, under some homeomorphic identification of $S^2$ with $\hat \C$, the map is equal to a rational map.  It is a fact that the rational maps of $\hat \C$ are exactly the holomorphic maps.  

Thurston observed that a strong reduction system is an obstruction to holomorphicity for a non-exceptional dynamical branched cover.  We will return to this point after the statement of Thurston's characterization of rational maps.  Because of Thurston's observation, strong reduction systems for non-exceptional dynamical branched covers are called Thurston obstructions in the literature.  Since strong reduction systems are not obstructions to holomorphicity in the exceptional cases, we avoid this terminology.

\begin{theorem}[Thurston's characterization of rational maps]
\label{thm:thu}
Let $f\colon \Sigma \to \Sigma$ be an unmarked dynamical branched cover where $\Sigma = (S^2,P)$. If $f$ is not exceptional, then $f$ is isotopic to a dynamical branched cover of one of the following two types:
\begin{enumerate}
    \item rational, or
    \item strongly reducible. 
\end{enumerate}
The two types are exclusive.  If $f$ is of type (1), the complex structure is unique up to isotopy.
\end{theorem}

Our statement of Thurston's characterization is different from, but equivalent to, the usual statement.  One difference is that our statement involves a stable multicurve instead of an invariant multicurve.  So in terms of finding an obstruction to rationality, our statement is stronger.  Another difference is that our statement makes no reference to a matrix or an eigenvalue (the labels on the strong reduction system play the role of the eigenvector).  

Pilgrim \cite{pilgrim03} showed that we can use Thurston's characterization of rational maps to say that every unmarked dynamical branched cover of $(S^2,P)$ reduces into pieces that are geometric, that is, rational.  This is analogous to the story for surface homeomorphisms, as above.

The uniqueness statement in Theorem~\ref{thm:thu} is often referred to as Thurston rigidity.  Hence the common parlance: Thurston's theorem states that a Thurston map has a Thurston obstruction---meaning that the Thurston matrix has a Thurston eigenvalue greater than or equal to 1---or it is Thurston equivalent to a rational map, which moreover satisfies Thurston rigidity.

\p{Topological polynomials, Levy cycles, and Levy--Berstein} We say that a dynamical branched cover $f : (S^2,P) \to (S^2,P)$ is a topological polynomial if $P$ contains a fixed point $p$ for $f$ and the local degree of $f$ at $p$ is equal to $\deg f$.  We may regard the topological polynomial $f$ as a dynamical branched cover of $(\R^2,P \setminus p)$ (so $p$ plays the role that $\infty$ plays for a polynomial).  Examples of topological polynomials include polynomials acting on $\hat \C$ (equivalently, acting on~$\C$).  

A multicurve $\{ \gamma_1, \dots, \gamma_k \}$ for a dynamical branched cover $f$ is a Levy cycle if there is a cyclic permutation $\sigma$ of $\{1,\dots,k\}$ so that for every $i$ there is a component $\tilde \gamma_i$ of $f^{-1}(\gamma_i)$ that is homotopic to $\gamma_{\sigma(i)}$ and that maps with degree 1 onto $\gamma_i$.  A Levy cycle is degenerate if each $\gamma_i$ bounds an embedded disk $\Delta_i$ so that for every $i$ some component of $f^{-1}(\Delta_i)$ that is homotopic to $\Delta_{\sigma(i)}$ and maps with degree 1 onto $\Delta_i$.

By work of Berstein, Hubbard, Levy, Rees, Tan, and Shishikura \cite[Theorem 10.3.7]{hubbard} we have the following refinement of Thurston's characterization of rational maps: \emph{a topological polynomial is either rational or it has a degenerate Levy cycle}.  In Appendix~\ref{sec:levy} we state and prove a strengthening of Levy's theorem, Proposition~\ref{prop:levy}.

Levy cycles are strong reduction systems.  However, they are not always Thurston obstructions since they are not always invariant multicurves.  It is a feature of our statement of Thurston's characterization of rational maps that Levy cycles suffice to obstruct rationality.  

Levy and Berstein give a sufficient criterion for a topological polynomial to be rational: each point of $P$ contains a critical point in its forward orbit.  This result is known as the Levy--Berstein theorem.  In Appendix~\ref{sec:levy} we explain how to derive this statement from Proposition~\ref{prop:levy}.  

\p{Portraits, homotopy, and Thurston equivalence} Above, we defined two dynamical branched covers $f$ and $g$ of $\Sigma$ to be homotopic if there is a homeomorphism $h$ of $\Sigma$ that is homotopic to the identity (rel $P$) and satisfies $f \circ h = g$.  We give here an alternate description of homotopic maps and also compare the notion of homotopy to the more commonly used notion of Thurston equivalence.  For the former we require the notion of an extended portrait.  

The portrait of a dynamical branched cover $f$ is the directed, labeled graph whose vertices are the post-critical points of $f$ and where there is an edge labeled $k$ from $p_1$ to $p_2$ if $f$ maps $p_1$ to $p_2$ with local degree $k$.  The extended portrait of $f$ is defined in the same way, except that the vertex set consists of the critical points and the post-critical points of $f$.

We may say that two dynamical branched covers of $\Sigma$ are homotopic if they are connected by a homotopy of maps $f_t : \Sigma \to \Sigma$ rel $P$ where each $f_t$ is a dynamical branched cover and all of the $f_t$ have the same extended portraits up to labeled, directed graph isomorphism.  This notion agrees with the notion of homotopy given earlier.  

Let $\Sigma = (S,P)$ and $T = (S,Q)$ be two marked surfaces.  In the literature, dynamical branched covers $f : \Sigma \to \Sigma$ and $g : T \to T$ are said to be Thurston equivalent (or combinatorially equivalent) if there are homeomorphisms $h_0,h_1 : \Sigma \to T$ that are homotopic (rel $P$) and satisfy $f \circ h_0 = h_1 \circ g$.  If, for example, $f$ and $g$ are polynomials with different post-critical sets, then it does not make sense for $f$ and $g$ to be homotopic, but it does make sense for them to be Thurston equivalent.  Because of this, Thurston's characterization of rational maps is usually stated in terms of Thurston equivalence.  We will not discuss Thurston equivalence in what follows.

\p{Orbifolds and Thurston obstructions} Let $\hat{\mathbb{N}}$ denote $\mathbb{N} \cup \{\infty\}$.  For our purposes, a (2-dimensional) orbifold is a marked surface $(S,P)$ endowed with a function $\nu : P \to \hat{\mathbb{N}}$.  We think of the function $\nu$ as a labeling of the points of $P$ by elements of $\hat{\mathbb{N}}$.

To a dynamical branched cover $f : (S,P) \to (S,P)$ there is an associated orbifold structure on $(S,P)$---that is, an associated function $\nu$---defined as follows.  For each $k$ and each critical point $c$ of $f^k$ with $f^k(c)=p$, we compute the local degree of $f^k$ at $c$.  The label $\nu_p$ is the least common multiple of these local degrees over all such choices of $k$ and $c$ (we take the least common multiple of the empty set to be 1, so the label on a non-postcritical point is 1).  We provide geometric meaning to this notion in Appendix~\ref{sec:orb}.  Briefly, the orbifold for $f$ is the minimal orbifold structure for which $f$ is a partial self-cover (in the orbifold sense).  Every orbifold falls into one of three categories---spherical, Euclidean, or hyperbolic---according to whether its Euler characteristic is positive, zero, or negative; see the appendix.  

Thurston's characterization of rational maps can equivalently be stated in terms of orbifolds instead of exceptional maps.  There is a particular orbifold $(S^2,P)$, called the $(2,2,2,2)$-orbifold, where $|P|=4$ and $\nu(p)$ is equal to 2 for all $p \in P$.  In the appendix, we show that a dynamical branched cover $f$ of $(S^2,P)$ is exceptional if and only if the orbifold for $f$ is the $(2,2,2,2)$-orbifold.  As such, we obtain an alternate statement of Thurston's characterization, namely, that if the orbifold for a dynamical branched cover $f$ of $(S^2,P)$ is not the $(2,2,2,2)$-orbifold, then (up to homotopy) $f$ is either holomorphic or it has a strong reduction system. 

With this in mind, we may think of Thurston's characterization of rational maps as a statement about maps with hyperbolic orbifold, as opposed to a statement about non-exceptional maps.  Indeed, a slight weakening of Theorem~\ref{thm:thu} is that if $f$ has hyperbolic orbifold, then $f$ is rational if and only if it is not strongly reducible (the only weakening is that this version leaves out non-exceptional Euclidean maps).  The $(2,2,2,2)$-orbifold is the only Euclidean orbifold with four cone points.  Since there are no essential curves on an orbifold with three marked points, there are no strong reduction systems and so by Thurston's characterization all such dynamical branched covers are rational.  To summarize, the reasons why Thurston's dichotomy holds for maps with hyperbolic orbifold and non-exceptional maps with Euclidean orbifold are different: in the former case strong reduction systems are obstructions to holomorphicity, and in the latter case there are no strong reduction systems. 

In the Appendix~\ref{sec:orb}, we use orbifolds to explain why strong reduction systems are obstructions to holomorphicity for maps with hyperbolic orbifold.  Unlike previous proofs in the literature, our argument makes no reference to Teichm\"uller space or the pullback map.  Instead, it relies on the geometric characterization of the orbifold for a dynamical branched cover that seems to not appear in the literature but was surely known to Thurston.  As with the Nielsen--Thurston classification, we can therefore think of Thurston's characterization of rational maps as saying that the only obstruction to holomorphicity is the ``obvious'' one.

\subsection{The Nielsen--Thurston \"Ubertheorem}\label{subsec:uber}

Before stating the \"Ubertheorem,  we introduce affine exceptional maps, which will appear in the statement.  We think of these as being geometric representatives of homotopy classes of maps, in the same way that pseudo-Anosov and holomorphic maps are.  

\p{Affine exceptional maps} An unmarked affine exceptional map of $T^2$ is simply that: 
an exceptional map of $T^2$ (in other words, a map of degree greater than 1) that is unmarked and preserves  some affine structure on $T^2$.  Again, all unmarked exceptional maps of $T^2$ are homotopic to affine exceptional maps.  To translate this notion to the sphere case, we again need to go through the hyperelliptic involution.

Fix an affine structure on $T^2$ and choose a base point.  As above, there is an associated hyperelliptic involution $\iota$, one of whose fixed points is the base point.  All linear maps of $T^2$ are symmetric with respect to $\iota$ and hence descend to unmarked dynamical branched covers of $(S^2,P_0)$, the sphere with four marked points.  These are examples of unmarked affine exceptional maps of $(S^2,P_0)$ (there are four marked points but, as per the definition of an unmarked map, they are all post-critical).

More generally, if we take a linear map of $T^2$ and compose it with a rotation of $T^2$ by $\pi$ in either or both factors, we obtain an affine map of $T^2$ that descends to a map of $(S^2,P_0)$.  Any such map is an unmarked affine exceptional map of $S^2$.  While $S^2$ carries no affine structure, it does carry many singular affine structures: those arising from affine structures on $T^2$.  Affine maps of $S^2$ preserve these singular affine structures.  

A dynamical branched cover of a torus or a sphere with four marked points is an unmarked affine exceptional map if it is affine with respect to some choice of (singular) affine structure.  

A marked exceptional dynamical branched cover is affine if the corresponding unmarked map (obtained by forgetting the extra marked points) is affine.  In the case of the torus this means forgetting all the marked points, and in the case of the sphere this means forgetting all but four (all of which being post-critical).  We emphasize that a marked map is affine if the corresponding unmarked map is actually an affine map, not just homotopic to an affine map.

\p{Statement of the \"Ubertheorem} After stating two definitions, we will give the statement of the \"Ubertheorem and explain how to derive the previous two theorems as special cases.

A dynamical branched cover $f : \Sigma \to \Sigma$ of degree $d$ is \emph{holomorphic} if it is holomorphic with respect to some complex structure on $\Sigma$.  And $f$ is \emph{pseudo-Anosov} if there is a pair of transverse measured foliations $(\F^+,\F^-)$ on $\Sigma$ that is preserved by $f$ and satisfies
\[
f^{-1}(\F^+,\F^-)=(\lambda\sqrt{d}\, \F^+,\tfrac{\sqrt{d}}{\lambda}\,\F^-)
\]
for some $\lambda > 1$.  The singularities have the same restrictions as in the case of a pseudo-Anosov homeomorphism.  

\begin{ubertheorem}
\label{thm:main}
Let $f\colon \Sigma \to \Sigma$ be a dynamical branched cover. Then $f$ is isotopic to a map $\phi$ of one of the following types:
\begin{enumerate}
    \item holomorphic,
    \item strongly reducible, or 
    \item pseudo-Anosov.
\end{enumerate}
If $f$ is of type (1) and of type (2), then either $\deg f = 1$ or $f$ is affine exceptional.  If $f$ is of type (2) and of type (3) then $f$ is affine exceptional.  If $f$ is of type (3) then either $\deg f = 1$ or $f$ is affine exceptional. 

If $f$ is of type (1) and $f$ is a non-exceptional map with $\deg f > 1$, then the associated complex structure is unique up to isotopy.  If $f$ of type (3) then the associated pair of measured foliations is unique up to isotopy.  
\end{ubertheorem}

As mentioned, the \"Ubertheorem has the Nielsen--Thurston classification and Thurston's characterization of rational maps as special cases.  To see that the Nielsen--Thurston classification is the $\deg f=1$ case, we must use the following three facts about homeomorphisms of surfaces: (1) a holomorphic homeomorphism of a surface of negative Euler characteristic has finite order (and a holomorphic homeomorphism of the torus is homotopic to a map of finite order), (2) a strong reduction system is nothing other than a reduction system, and (3) a pseudo-Anosov dynamical branched cover of degree 1 is a pseudo-Anosov homeomorphism.

To obtain Thurston's characterization of rational maps from the \"Ubertheorem, we use the fact that holomorphic maps of $S^2$ are rational.  Since we do not require the marked points of $\Sigma$ to be post-critical, the \"Ubertheorem also implies the generalization of Thurston's characterization due to Buff--Cui--Tan, which extends the theorem to the case of marked dynamical branched covers \cite[Theorem 2.1]{BCT}.  

While we are not aware of any theorems in the literature that combine the exceptional cases of Thurston's characterization of rational maps into the classical statement, a result of Bartholdi--Dudko does give an analogue of the \"Ubertheorem for the exceptional cases themselves \cite[Theorem A]{BD}.

\p{Extensions of the \"Ubertheorem: non-orientable surfaces, orientation reversing maps, equivariant maps} In Appendix~\ref{sec:ext}, we explain how our argument for the Nielsen--Thurston \"Ubertheorem applies in even further generality.  Specifically, we give extensions to the cases of non-orientable surfaces and the cases of orientation-reversing dynamical branched covers.  We also give a version of the \"Ubertheorem for equivariant dynamical branched covers.

\p{The Bers strategy} As in the Bers proof of the Nielsen--Thurston classification, we prove the \"Ubertheorem by appealing to the action of a dynamical branched cover $f : \Sigma \to \Sigma$ on the Teichm\"uller space $\Teich(\Sigma)$.  A point in $\Teich(\Sigma)$ is an equivalence class of complex structures on $\Sigma$.  By pulling back complex structures through $f$, we obtain Thurston's pullback map $\sigma_f : \Teich(\Sigma) \to \Teich(\Sigma)$.  (In the original Bers proof, it makes sense to consider either pullback or push-forward, but for covers of higher degree only pullback makes sense in general.)

Following Bers, we consider the translation length $\tau$ of $\sigma_f$, that is, the infimum of the distances $d(X,\sigma_f(X))$ over all $X$ in $\Teich(\Sigma)$.  There are three cases for $\tau$: it can be 0 and realized, not realized, or nonzero and realized.  In the first case, $\sigma_f$ has a fixed point, which means that $f$ is holomorphic.  In the second case, we show that $f$ has a reduction system.  As in the original Bers proof, this is derived as a consequence of the Mumford compactness criterion.  When $\deg f > 1$ we augment the original Bers proof to show that there is an orbit for $\sigma_f$ that goes to infinity (towards the reduction system); this is the content of Proposition~\ref{prop:metricorbit}.  Then, assuming the reduction system is not strong, we show that this orbit is also repelled from infinity, a contradiction.  Finally, in the third case, we show that $\sigma_f$ preserves a geodesic ray in $\Teich(\Sigma)$.  This phenonenon, which does not seem to have been observed before for $\deg f > 1$, is elucidated in Proposition~\ref{prop:metricray}.  We show that this only occurs in the exceptional cases and the cases where $\deg f = 1$.    Perhaps unexpectedly, the usual discussion for the $\deg f =1$ applies in this more general case.  As in the original Bers proof, we then show that a geodesic ray corresponds to a pair of transverse measured foliations, and the translation distance along the ray corresponds to a stretch factor $\lambda$, thus implying that $f$ is pseudo-Anosov.

\p{Examples of non-exclusivity} Figure~\ref{fig:venn} gives examples of dynamical branched covers of $T^2$ of all the different types allowed by the \"Ubertheorem when the cover is exceptional and the degree is greater than 1: holomorphic, holomorphic and strongly reducible, strongly reducible, strongly reducible and Anosov, and Anosov.  (Here we say ``Anosov'' instead of ``pseudo-Anosov'' since the underlying surface is a torus, and hence the corresponding foliations have no singularities.)  For the first three examples, we require that $d$ be a perfect square. As demanded by the \"Ubertheorem, the strongly reducible and Anosov example fails to be Anosov when $d=1$.

\begin{center}
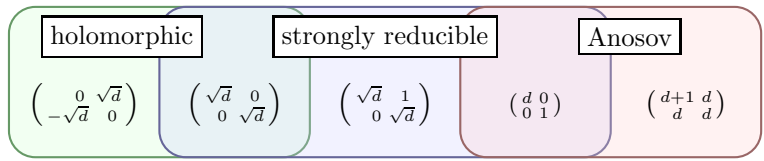
\begin{figure}[h!]
\begin{tikzpicture}
    \draw[green!20!gray,fill=green!20,fill opacity =.25, rounded corners=10,thick]
     (-4,0) rectangle (0,2);
\draw[blue!20!gray,fill=blue!20,fill opacity=.25,rounded corners=10,thick]
     (-2,0) rectangle (4,2);
     \draw[red!20!gray,fill=red!20,rounded corners=10,fill opacity=.25, thick]
     (2,0) rectangle (6,2);
\node at (-2.5,1.6) {\fcolorbox{black}{white}{\small holomorphic}};
\node at (1,1.6) {\fcolorbox{black}{white}{\small strongly reducible}};
\node at (4.25,1.6) {\fcolorbox{black}{white}{\small Anosov}};
\node at (-3,.7) {\Small $\left(\begin{smallmatrix}\ \  \  0&\sqrt{d} \\-\sqrt{d} & \ 0 \end{smallmatrix}\right)$};
\node at (-1,.7) {\Small $\left(\begin{smallmatrix} \sqrt{d} & \ 0 \\ \ 0 & \sqrt{d} \end{smallmatrix}\right)$};
\node at (1,.7) {\Small $\left(\begin{smallmatrix}\sqrt{d} & \ 1\\ \ \ 0 &\sqrt{d} \end{smallmatrix}\right)$};
\node at (3,.7) {\Small $\left(\begin{smallmatrix} d & 0 \\ 0 & 1 \end{smallmatrix}\right)$};
\node at (5,.7) {\Small $\left(\begin{smallmatrix}d+1 & d \\ d & d\end{smallmatrix}\right)$};
\end{tikzpicture}
\caption{A Venn diagram of different types of dynamical branched covers of~$T^2$}
\label{fig:venn}
\end{figure}
\end{center}

\p{Comparison to Douady--Hubbard}  The original proof of Thurston's characterization of rational maps is detailed in Douady--Hubbard's paper \cite{DH} and Hubbard's book \cite{hubbard}.  Our approach is the same in spirit, but differs in the following ways:
\begin{enumerate}
\item we appeal to Teichm\"uller's theorems instead of working with the derivative of the pullback map (our application of Teichm\"uller's uniqueness theorem is morally equivalent to Lemma~1 of Douady--Hubbard), 
\item we avoid explicit mention of hyperbolic surfaces, staying entirely in the category of Riemann surfaces,
\item we give a simplified treatment of the combinatorial topological step (Proposition~\ref{prop:stable}) and, like Buff--Cui--Tan, we directly address the case where there are marked points that are not post-critical (the cost of our simplification is the loss of sharpness), 
\item we isolate in Section~\ref{sec:synthetic} the basic properties of metric spaces we use, 
and
\item we clarify the role that orbifolds play in the proof that strong reduction systems are obstructions to holomorphicity in the case of a non-exceptional map.
\end{enumerate}
Another feature of our exposition is that we treat many cases of Thurston's characterization of rational maps that were not addressed before, namely, the cases where $S=T^2$, where $S$ is non-orientable, where $f$ reverses orientation, and where $f$ is equivariant with respect to a finite group action.  The arguments of Douady--Hubbard could similarly be extended to prove these additional cases.

One other philosophical difference between our approach and the prevailing literature is that we make no mention of Thurston equivalence.  To wit, instead of considering maps up to homotopy and conjugacy, we only consider maps up to homotopy.  This point of view has long been championed by Kevin Pilgrim.  

We emphasize that there is a general translation between the Douady--Hubbard proof and our proof; and in the text that follows we have indicated the points of similarity.  We hope that our exposition will appeal to those already familiar with the Bers proof of the Nielsen--Thurston classification theorem, and will also clarify the relationship between that theorem and Thurston's characterization of rational maps.

Work in progress by Drach--Reinke--Schleicher \cite{DS} gives a new approach to the four theorems of Thurston involving the pullback map (two of which are the ones discussed in this paper).  Their approach also uses Teichm\"uller's theorems instead of the derivative of the pullback map.

\p{Latt\`es maps and Euclidean maps} The exceptional maps that we consider overlap with several other notions in the literature, and the terminology is used differently by different authors.  A Latt\`es map is a holomorphic branched cover $S^2\rightarrow S^2$ that is the finite quotient of a holomorphic affine map of $T^2$.  Milnor gives a thorough survey and further characterization of Latt\`es maps \cite{milnor}.  A Latt\`es-type map is a (not-necessarily-holomorphic) quotient of an affine map of $T^2$ (this is not typically given as the definition of Latt\`es-type, but Bonk--Meyer prove that it is equivalent \cite[Theorem 1.2]{BM}). The exceptional maps we consider are Latt\`es-type maps where the finite quotient is by the hyperelliptic involution.  (Milnor also defines finite quotients of affine maps, which have a similar definition as a Latt\`es map, except with the torus possibly replaced by a cylinder; these types of maps do not arise in this paper.)

Cannon--Floyd--Parry--Pilgrim consider Euclidean maps, which they define as dynamical branched covers of $S^2$ with at most four post-critical points, none of which are critical, such that every critical point is simple (local degree two) \cite{CFPP}.  These are precisely the dynamical branched covers with Euclidean orbifold and at least four (hence exactly four) post-critical points.  Our exceptional maps of $S^2$ are the Euclidean maps of Cannon--Floyd--Parry--Pilgrim.  Cannon--Floyd--Parry--Pilgrim also introduce and study nearly Euclidean maps, which are branched covers of $S^2$ with exactly four post-critical points and where each critical point is simple (such as the rabbit polynomial).

\subsection{Overview of the paper} We divide the proof of the \"Ubertheorem into five parts, each with their own section.  The first three of these sections isolate three different aspects of the proof, namely, combinatorial topology, Teichm\"uller theory, and metric space theory.  Sections~\ref{sec:proof} and~\ref{sec:torus} tie these together to prove the theorem for the non-exceptional and exceptional cases, respectively.  While the exceptional cases are handled separately, we emphasize that the proof is essentially the same; the main content is already contained in the non-exceptional case, while the exceptional case requires a few extra technical details.

In Section~\ref{sec:rh}, we give a combinatorial topological statement, Proposition~\ref{prop:stable}.  It says that, under certain hypotheses on $f$, at most 3 marked points have the property that all of their iterated preimages under $f$ are critical or marked.

In Section~\ref{sec:pullback} we prove Proposition~\ref{prop:pullback}, which is about the pullback map on Teichm\"uller space~$\sigma_f$.  The proposition states that if $\deg f > 1$ and if $f$ is not exceptional, then some iterate of $\sigma_f$ is weakly contracting, meaning that it decreases the distance between all pairs of points.  The proof uses  Proposition~\ref{prop:stable}.

In Section~\ref{sec:synthetic} we prove three statements about metric spaces, namely, Propositions~\ref{prop:metricmumford}, \ref{prop:metricorbit}, and~\ref{prop:metricray}.  The purpose is to isolate the parts of the proof of the \"Ubertheorem that only use the theory of metric spaces and not the theory of Teichm\"uller space.

In Section~\ref{sec:proof} we follow the Bers proof of the Nielsen--Thurston classification in order to prove the \"Ubertheorem in the non-exceptional cases.  Our argument follows the Bers strategy described above.  Again, the key idea is to consider the translation length $\tau$ of the pullback map on Teichm\"uller space and separately investigate the three cases where $\tau$ is 0 and realized, nonzero and realized, and not realized.  These three cases exactly correspond to the three cases in the conclusion of the \"Ubertheorem.

Finally in Section~\ref{sec:torus}, we prove the \"Ubertheorem in the exceptional cases.  We prove that in these cases the associated Teichm\"uller space decomposes as a product in a natural way, and apply the ideas of Section~\ref{sec:proof} to the action of a dynamical branched cover on the product structure.  Among maps of degree greater than 1, the exceptional $f$ that preserve a horizontal slice are exactly the ones whose associated pullback maps fail to have weakly contracting orbits.  This is the reason why exceptional maps require separate consideration.  

There are three appendices.  Appendix~\ref{sec:orb} gives a direct proof of the fact that strong reduction systems are obstructions to holomorphicity for dynamical branched covers with hyperbolic orbifold.  Along the way, we clarify the geometric meaning of the orbifold structure for a dynamical branched cover.  In Appendix~\ref{sec:levy} we explain how Thurston's characterization of rational maps specializes in the case of topological polynomials.  Appendix~\ref{sec:ext} describes how our arguments apply to give generalizations of the \"Ubertheorem to the cases of equivariant dynamical branched covers, dynamical branched covers of non-orientable surfaces, and orientation-reversing dynamical branched covers.

\p{Acknowledgments} We would like to thank Wolf Jung, Jeremy Kahn, Sanghoon Kwak, Yair Minsky, Insung Park, Kevin Pilgrim, Dierk Schleicher, Roberta Shapiro, and Sam Taylor for helpful comments and conversations.  The second author is grateful to the Georgia Institute of Technology for supporting this work.  The third author is grateful to the Mathematical Sciences Research Institute for a stimulating work environment.


\section{Stable marked points}
\label{sec:rh}

The goal of this section is to prove Proposition~\ref{prop:stable} below.  This is the main ingredient in the proof of Proposition~\ref{prop:pullback} in Section~\ref{sec:pullback}.   A refined version of this proposition is given by Lemma~2 of Douady--Hubbard.  To state our proposition, we require the notion of stability.

\p{Stability of marked points} Let $\Sigma = (S,P)$ and let $f : \Sigma \to \Sigma$ be a dynamical branched cover.  We say that $p \in P$ is stable if $f^{-1}(p) \subseteq P \cup \Crit(f)$.  We say that $p$ is infinitely stable if $f^{-k}(p) \subseteq P \cup \Crit(f^k)$ for all $k \geq 0$.

If $f$ is exceptional, then each post-critical point is infinitely stable. The following proposition is a sort of converse to this statement.

\begin{proposition}
\label{prop:stable}
Let $\Sigma = (S^2,P)$, and let $f : \Sigma \to \Sigma$ be a dynamical branched cover of degree $d > 1$.  If $f$ is not exceptional, then $f$ has fewer than 4 infinitely stable marked points.
\end{proposition}

\begin{proof}

Let $Q \subseteq P$ be the set of infinitely stable points for $f$, and suppose that $|Q|\geq 4$.  We will show that $f$ is exceptional.

Let $\tilde{Q} = f^{-1}(Q)$, and let $C = \Crit(f) \cap \tilde Q$.  If a non-critical marked point maps to an infinitely stable marked point, then it itself is infinitely stable, that is, 
$\tilde{Q}\subseteq Q\cup C$.
In particular,
\[
|\tilde{Q}| \leq |C| + |Q|.
\]
Since (counting with multiplicity) a critical point of degree $k$ accounts for $k$ pre-images of a point in $Q$, we also have
\[
|\tilde Q|  = |Q|d - \sum_{c\in C} \bigl( \deg_f(c) - 1\bigr).
\]
By the Riemann--Hurwitz formula and the preceding equality and inequality we have
\begin{align*}
2d-2\geq \sum_{c\in C}\bigl(\deg_f(c)-1\bigr) = |Q|d-|\tilde Q|  \geq |Q|d-|Q|-|C| = |Q|(d-1)-|C|.
\end{align*}
We conclude that $|C| \geq (|Q|-2)(d-1)$.  Since $d > 1$ and a branched cover $S^2 \to S^2$ of degree $d$ has at most $2d-2$ critical points, it follows that $|Q|\leq 4$.  By our earlier assumption that $|Q| \geq 4$, we conclude that $|Q|=4$. 

Replacing $|Q|$ with 4 in the above, we conclude that $|C|=2d-2$, so $C$ is equal to all of $\Crit(f)$ and each critical point is simple.  Moreover, the inequality must be an equality, so in particular $|\tilde{Q}|=|C|+|Q|$, which means that $C$ is disjoint from $Q$ and $Q\subseteq \tilde{Q}$.  This means that $f(Q)\subseteq Q$.  Since $Q$ contains all the critical values of~$f$, it follows that $Q$ contains the post-critical set.  

Because the preimage of each point of $Q$ is either in $Q$ or in $C$ it follows that every point in $Q$ must be post-critical.  Since the critical points are all simple, the ramification index at each point of $Q$ is 2.  In other words, the orbifold for $f$ is the $(2,2,2,2)$-orbifold.  As in the introduction, this is equivalent to the statement that $f$ is exceptional, as desired.
\end{proof}

Similar arguments can be used to derive a stronger conclusion if $P$ is the post-critical set: the second iterate $f^2$ must have fewer than $4$ stable marked points, and if $f$ is a topological polynomial then $f$ itself must have fewer than $4$ stable marked points.  Combining this with the proof of Proposition~\ref{prop:pullback} below, it follows that $\sigma_f^2$ is weakly contracting whenever $P$ is the post-critical set, and $\sigma_f$ is weakly contracting in this case if $f$ is a topological polynomial.


\section{Pullback is a weak contraction}
\label{sec:pullback}

The goal of this section is to prove Proposition~\ref{prop:pullback}, which states that the pullback map is non-expanding, and in many cases weakly contracting.  A refinement of this statement is given in Proposition~3.3 of  Douady--Hubbard.  Both of these results are in concert with a theorem of Royden, which says that analytic maps of Teichm\"uller space are weak contractions \cite{royden}.  As in the work of Douady--Hubbard, we will neither use the analyticity of the pullback map nor the Royden result. We begin with the requisite definitions; see \cite[Chapter 11]{primer} for more details.

\p{Teichm\"uller space and the pullback map} Let $\Sigma = (S,P)$ and let $f : \Sigma \to \Sigma$ be a dynamical branched cover.  The Teichm\"uller space $\Teich(\Sigma)$ is the set of complex structures on $\Sigma$ up to isotopy.  More specifically, a complex structure on $\Sigma$ is a complex structure on $S$ and two complex structures $X$ and $Y$ on $\Sigma$ are equivalent if there is a isomorphism $h : X \to Y$ that is isotopic to the identity (here we insist that $h(P)=P$ and that isotopies fix $P$).

The pullback map associated to $f$ is the map
\[
\sigma_f : \Teich(\Sigma) \to \Teich(\Sigma)
\]
defined by pulling back complex structures through $f$.  

\p{The Teichm\"uller metric and Teichm\"uller's theorems} The Teichm\"uller metric on $\Teich(\Sigma)$ is defined as follows.  For a map $h$ between Riemann surfaces, let $K(h)$ denote the quasi-conformal dilatation. Given $X,Y \in \Teich(\Sigma)$ we set
\[
K(X,Y) = \inf \{K(h) \mid h : X \to Y \text{ and } h \sim \textrm{id} \}
\]
and
\[
d(X,Y) = \frac{1}{2} \log K(X,Y).
\]
Teichm\"uller's existence theorem gives that the infimum is a minimum, that is, there is a map~$h$, called the Teichm\"uller map, that realizes the infimum \cite[Theorem 11.8]{primer}.  Teichmuller's uniqueness theorem states that the minimizing map $h$ is unique \cite[Theorem 11.9]{primer}. 

\p{Teichm\"uller maps and foliations} Teichm\"uller's existence theorem further gives an explicit description of the Teichm\"uller map~$h$.  Usually, this description is phrased in terms of quadratic differentials.  We avoid this terminology here.

For the description of $h$, we need the fact that a pair of transverse measured foliations $(\F^+,\F^-)$ on $\Sigma$ induces a complex structure on $\Sigma$; in other words, $(\F^+,\F^-)$ represents a point in $\Teich(\Sigma)$.  Indeed, $(\F^+,\F^-)$ gives a Euclidean structure on $\Sigma$ away from the singularities, and hence (orientation-preserving) charts to the complex plane, well-defined up to rotation.  If the charts identify segments of the leaves of $\F^+$ and $\F^-$ with horizontal and vertical line segments, then they are called natural coordinates for $(\F^+,\F^-)$.  These are well defined up to translation in~$\C$.

Now, Teichm\"uller's description of the Teichm\"uller map $h$ is that there is a pair of measured foliations $(\F^+,\F^-)$ on $\Sigma$ so that, setting $\lambda=\sqrt{K(h)}$, we have
\begin{itemize}
\item $(\F^+,\F^-)$ induces $X$,
\item $(\lambda \,\F^+,\tfrac{1}{\lambda}\,\F^-)$ induces $Y$, and
\item in natural coordinates with respect to these two pairs of foliations, $h$ is given by 
\[
\left(\begin{array}{cc}
  \lambda   & 0 \\
  0   & 1/\lambda
\end{array}\right)
\]
\end{itemize}

One way to rephrase Teichm\"uller's theorems is that every geodesic ray in $\Teich(\Sigma)$ is determined by a pair of transverse measured foliations $(\F^+,\F^-)$ on $\Sigma$, and the ray is obtained by multiplying $\F^+$ by $\lambda \geq 1$ and $\F^-$ by $1/\lambda$.

The measured foliations $\F^+$ and $\F^-$ must have singularities if $\chi(S) \neq 0$.  If there are any 1-pronged singularities, they must be at points of $P$, for otherwise $K(h)$ is not minimal.

\p{The pullback map is non-expanding or weakly contracting} Let $(T,d)$ be a metric space and let $\sigma : T \to T$.  We say that $\sigma$ is non-expanding if
\[
d(\sigma(x),\sigma(y)) \leq d(x,y)
\]
for all $x,y \in T$.  We say that $\sigma$ is weakly contracting if 
\[
d(\sigma(x),\sigma(y)) < d(x,y)
\]
for all distinct $x,y \in T$.  

\begin{proposition}
\label{prop:pullback}
Let $\Sigma = (S,P)$, and let $f: \Sigma \to \Sigma$ be a dynamical branched cover.
\begin{enumerate}[label={(\arabic*)},ref={\theproposition(\arabic*)}]
    \item\label{nonexp} The pullback map $\sigma_f$ is non-expanding.
    \item If $f$ is not exceptional and $\deg(f) > 1$, then $\sigma_f^k$ is weakly contracting for some $k \geq 1$.
\end{enumerate}
\end{proposition}

\p{Idea of the proof and pseudo-Teichm\"uller maps} Before proving Proposition~\ref{prop:pullback}, we explain the main observation used in the proof.  Let $f : \Sigma \to \Sigma$ be a dynamical branched cover, let $X,Y \in \Teich(\Sigma)$, and let $h : X \to Y$ be a Teichm\"uller mapping.  Since $h$ is isotopic to the identity, there is a unique map $h^f$, which we call the lifted map, that is isotopic to the identity and so that the following diagram commutes:
\[
 \begin{tikzcd}
  \sigma_f(X) \arrow[r,"h^f"] \arrow[d,"f"] & \sigma_f(Y) \arrow[d,"f"] \\
 X \arrow[r,"h"]  & Y
 \end{tikzcd}
\]
We can incorporate the pair of foliations $(\F^+,\F^-)$ into the diagram:
\[
 \begin{tikzcd}
  (\sigma_f(X),f^*(\F^+,\F^-)) \arrow[r,"h^f"] \arrow[d,"f"] & (\sigma_f(Y),f^*(\lambda\,\F^+,\tfrac{1}{\lambda}\,\F^-)) \arrow[d,"f"] \\
 (X,(\F^+,\F^-)) \arrow[r,"h"]  & (Y,(\lambda\,\F^+,\tfrac{1}{\lambda}\,\F^-))
 \end{tikzcd}
\]
As the pullback $f^*(\lambda\,\F^+,\tfrac{1}{\lambda}\,\F^-)$ on the top right of the diagram is equal to $(\lambda\, f^*(\F^+),\tfrac{1}{\lambda}\,f^*(\F^-))$
the map $h^f$ has the same quasiconformal dilatation as $h$.  It locally behaves like a Teichm\"uller map whose associated foliations are the pullbacks of the foliations for $h$.  However, $h^f$ need not be a Teichm\"uller map, because it is possible that these foliations have 1-pronged singularities at unmarked preimages of points of $P$.  

In general, if a map is obtained from a Teichm\"uller map by forgetting a marked point at one of the associated 1-pronged singularities, we call that map a pseudo-Teichm\"uller mapping.  The key point is that pseudo-Teichm\"uller mappings are not themselves Teichm\"uller mappings.

\begin{proof}[Proof of Proposition~\ref{prop:pullback}] 

Let $X,Y \in \Teich(\Sigma)$.  Let $h : X \to Y$ be the Teichm\"uller map, which exists by Teichm\"uller's existence theorem.  As above, the lifted map
\[
h^f : \sigma_f(X) \to \sigma_f(Y)
\]
is a Teichm\"uller map or pseudo-Teichm\"uller map with the same quasi-conformal dilatation as~$h$.  The first statement follows now from the definition of the Teichm\"uller metric.  

Suppose now that $f$ is not exceptional and $\deg(f) > 1$.  In this case $S=S^2$ and $f$ is not the quotient of an affine map by the hyperelliptic involution.  Since $S=S^2$, the foliations associated to $h$ must have at least four 1-pronged singularities at points of $P$. By Proposition~\ref{prop:stable}, there is a $k$ so that at least one of these four points of $P$ fails to be stable for $f^k$.  Therefore the pulled back map
\[
h^{f^k} : \sigma_f^k(X) \to \sigma_f^k(Y)
\]
is a pseudo-Teichm\"uller map and not a Teichm\"uller map.  The second statement follows from Teichm\"uller's uniqueness theorem and the definition of the Teichm\"uller metric.
\end{proof}

As mentioned, the analogue of Proposition~\ref{prop:pullback} in Douady--Hubbard is their Proposition~3.3.  The key to that proof is their Lemma~1, which is the analogue of our observation that the pullback of a Teichm\"uller map is a pseudo-Teichm\"uller map.  There they observe that the pullback of a Beltrami differential $q$ has norm greater than or equal to that of $q$, and that we have equality if and only if the preimages of the images of the poles of $p$ are critical or post-critical.  Through the duality between equivalence classes of Beltrami differentials (tangent vectors for Teichm\"uller space) and holomorphic quadratic differentials (cotangent vectors for Teichm\"uller space), we see that the two arguments are essentially the same.  Indeed, a Beltrami differential can be thought of as an ellipse field, and there is a natural ellipse field associated to a Teichm\"uller map.  In this way, our argument using Teichm\"uller's theorems recovers the Douady--Hubbard statement that the derivative of (an iterate of) the pullback map is contracting \cite[Proposition 3.3]{DH}.


\section{Synthetic Nielsen--Thurston theory}
\label{sec:synthetic}

By a synthetic Nielsen--Thurston package, we mean a collection $(T,P,\phi,\sigma)$, where 
\begin{enumerate}
    \item $T$ is a uniquely geodesic metric space where all maximal geodesics are bi-infinite,
    \item $P$ is a group acting properly discontinuously on $T$, 
    \item $\phi : P \dasharrow P$ is a virtual endomorphism, and
    \item $\sigma : T \to T$ is a function that is intertwined with $\phi$ and is non-expanding.
\end{enumerate}
Here a \emph{virtual endomorphism} $\phi : P \dasharrow P$ is a homomorphism $L \to P$ where $L$ is a finite-index subgroup of $P$.  We say $\sigma$ is \emph{intertwined} with $\phi$ if $\sigma(g\cdot x) = \phi(g)\cdot \sigma(x)$ for all $x \in T$ and $g \in L$.

In this paper, the only synthetic Nielsen--Thurston packages we will consider are ones where the space $T$ is $\Teich(\Sigma)$ for some marked surface $\Sigma$, where $P$ is the pure mapping class group $\PMod(\Sigma)$, where $\phi$ is the lifting homomorphism associated to a given dynamical branched cover $f : \Sigma \to \Sigma$ (see Section~\ref{sec:proof}), and where $\sigma$ is the pullback map $\sigma_f$. Our axiomatic approach is meant to clarify which properties of these objects are essential for the argument.

We will write $\tau_\sigma(X)$ for $d(X,\sigma(X))$ and $\tau_\sigma$ for the translation distance, which is the infimum of $\tau_\sigma(X)$ over $X \in T$:
\[
\tau_\sigma = \inf_{X \in T} \tau_\sigma(X).
\]
In this section we prove three propositions about synthetic Nielsen--Thurston packages, Propositions~\ref{prop:metricmumford}, \ref{prop:metricorbit}, and~\ref{prop:metricray}.  These will be used in the proof of the Nielsen--Thurston \"Ubertheorem to address the cases where
\begin{enumerate}
    \item $\tau_\sigma$ is not realized and $\sigma_f$ is non-expanding,
    \item $\tau_\sigma$ is not realized and $\sigma_f$ is weakly contracting, and
    \item $\tau_\sigma$ is realized and $\sigma_f$ is non-expanding.
\end{enumerate}
In the proof of the \"Ubertheorem in Section~\ref{sec:proof}, these appear in Case 2 ($\deg f=1$ subcase), Case 2 ($\deg f>1$ subcase), and Case 3, respectively.

\p{Translation distances not realized} The following proposition is a slight generalization of one of the steps in the Bers proof of the Nielsen--Thurston classification \cite[Section 13.6.1, Step 1]{primer}.  In that classical setting, the map $\phi$ is simply the inner automorphism of the mapping class group corresponding to $f^{-1}$ (this makes sense because the lift of a homeomorphism $g$ under a homeomorphism $f$ is $f^{-1}gf$).

\begin{proposition}
\label{prop:metricmumford}
Let $(T,P,\phi,\sigma)$ be a synthetic Nielsen--Thurston package where $\tau_\sigma$ is not realized.  If $\{X_n\}$ is a sequence in $T$ with
\[
\tau_\sigma(X_n) \to \tau_\sigma,
\]
then the image of $\{X_n\}$ in $T/P$ is not contained in any compact set.
\end{proposition}

\begin{proof}

Suppose to the contrary that the image of $\{X_n\}$ has compact closure.  We will find a point $Z$ so that $\tau_\sigma(Z) \leq \tau_\sigma$, contrary to the assumption that $\tau_\sigma$ is not realized.  

Let $L$ be the domain of~$\phi$, and let $\pi : T \to T/L$ be the quotient map.  Since $L$ has finite index in $P$, the map $T/L \to T/P$ is finite-to-one.   Thus $\{\pi(X_n)\}$ has a limit point, which is $\pi(Y)$ for some $Y \in T$.  

The desired $Z$ will be in the $L$-orbit of $Y$.  To find this $Z$, we define $F \colon T/L \to[0,\infty)$ by
\[
F(\pi(X)) = \min_{g \in L} \tau_\sigma(g \cdot X).
\]
We will prove below that $F$ is well defined, which implies two further statements: 
\begin{enumerate}
\item $F$ is continuous, and
\item there exists $g \in L$ with  $\tau_\sigma(g \cdot Y) \leq \tau_\sigma$ $\Longleftrightarrow$ $F(\pi(Y)) \leq \tau_\sigma$.
\end{enumerate}
Moreover, the last inequality follows from the continuity of $F$ and the definition of $Y$.  

It remains to prove that $F$ is well defined.  To this end, we give another description of $F$.  Using the definition of $\tau_\sigma(g \cdot X)$, the assumption that $\sigma$ is intertwined with $\phi$, and the fact that elements of $L$ act by isometries on $\Teich(\Sigma)$, we have
\[
\tau_\sigma(g\cdot X) = d\bigl(g\cdot X,\sigma(g\cdot X)\bigr) = d\bigl(g\cdot X,\phi(g)\cdot \sigma(X)\bigr) =  d\bigl(X,g^{-1}\phi(g)\cdot \sigma(X)\bigr).
\]
From this we obtain the following description of $F$:
\[
F(\pi(X)) = \min_{g \in L} \, d(X,g^{-1}\phi(g) \cdot \sigma(X)).
\]
The set of points $g^{-1}\phi(g)\cdot \sigma(X)$ is a subset of the $P$-orbit of $\sigma(X)$.  Since $P$ acts properly discontinuously, the given minimum exists, which is to say $F$ is well defined.  
\end{proof}

We remark that the proof of Proposition~\ref{prop:metricmumford} does not use the non-expanding property of $\sigma$.


\p{Weakly contracting orbits} The next proposition is essentially the same as Proposition~5.1 of Douady--Hubbard.  We begin by giving the definition of a weakly contracting orbit.

Given a self-map $\sigma$ of a metric space $T$ and an orbit $\mathcal{O} = (X_i)_{i=1}^\infty$ where $X_i=\sigma^i(X)$, we say that $\mathcal{O}$ is weakly contracting if the sequence $d(X_i,X_{i+1})$ is strictly decreasing (in particular, no two $X_i$ are equal).  Since $d(X_{i+1},X_{i+2})$ is equal to $d(\sigma(X_i),\sigma(X_{i+1}))$, it follows that all orbits of a weakly contracting map are weakly contracting.  It also follows from the definitions that if all orbits of a map are weakly contracting, then the map has no fixed points.

\begin{proposition}
\label{prop:metricorbit}
Let $(T,P,\phi,\sigma)$ be a synthetic Nielsen--Thurston package.  If every orbit for $\sigma$ is weakly contracting, then every orbit leaves every compact subset of $T/P$.  
\end{proposition}

Note that Proposition~\ref{prop:metricorbit} applies whenever $\sigma$ is weakly contracting and $\tau_\sigma$ is not realized, since having a fixed point implies that $\tau_\sigma$ is realized (and is equal to 0).

\begin{proof}[Proof of Proposition~\ref{prop:metricorbit}]

Let $\mathcal{O}=(X_i)$ be a $\sigma$-orbit.  Suppose for the sake of contradiction that the image of $\mathcal{O}$ in $T/P$ has compact closure.  In order to obtain a contradiction, we will find another $\sigma$-orbit $(Y_i)$ whose first three terms satisfy
\[
d(Y_0,Y_1) = d(Y_1,Y_2).
\]
Here is why this is a contradiction.  Since $Y_i$ is a $\sigma$-orbit, the above equality is equivalent to
\[
d(Y,\sigma(Y))=d(\sigma(Y),\sigma^2(Y))
\]
where $Y=Y_0$; equivalently, $\tau_\sigma(Y) = \tau_\sigma(\sigma(Y))$.  By the weakly contracting property of $\sigma$, this implies that $d(Y,\sigma(Y))=0$, which is to say that $\sigma$ has a fixed point, contrary to the assumption that all orbits of $\sigma$ are weakly contracting.  

To find such a $Y=Y_0$, our strategy is similar to the one used in the proof of Proposition~\ref{prop:metricmumford}.  Because we need to analyze three consecutive points in an orbit, instead of just two, we need to replace $T/L$ with a further finite cover of $T/P$.  To this end, let
\[
L_2 =  \{g \in P \mid \phi^2(g)\text{ is defined}\}.
\]
The subgroup $L_2$ has finite index in~$P$. Let $\pi$ be the quotient map
\[
\pi : T \to T/L_2.
\]
Since $L_2$ has finite index, the sequence $\{\pi(X_i)\}$ has a limit point, which is $\pi(Y)$ for some $Y \in T$.  We will show that, up to replacing $Y$ with another point in its $L_2$-orbit, $\tau_\sigma(Y) = \tau_\sigma(\sigma(Y))$. 

First we define a function $F$, analogous to the one in the proof of Proposition~\ref{prop:metricmumford}.  Since the sequence $\tau_\sigma(X_i) = d(X_i,X_{i+1})$ is non-negative and strictly decreasing (by the weakly contracting assumption), it converges to some $\delta\geq 0$.  We define $F\colon T/L_2 \to [0,\infty)$
by
\[F\bigl(\pi(X)\bigr) =  \min_{g\in L_2} \left\{ \bigl|\tau_\sigma(g \cdot X) - \delta\bigr| \ + \  \bigl|\tau_\sigma(\sigma(g \cdot X)) - \delta\bigr| \right\}.
\]
Assuming $F$ is well defined we have
\[
F\bigl(\pi(X_i)\bigr) \leq \left|\tau_\sigma(X_i) -\delta\right| \ + \  \left|\tau_\sigma(X_{i+1}) -\delta \right|
\]
for each $i$.  It follows that $F\bigl(\pi(X_i)\bigr) \to 0$.

We now use $F$ to analyze $Y$.  Again assuming $F$ is well defined, it is continuous.  Therefore, the statement $F\bigl(\pi(X_i)\bigr) \to 0$ implies that $F\bigl(\pi( Y) \bigr)=0$. Thus, after possibly replacing $Y$ with a different point in its $L_2$-orbit, we have 
\[
 \bigl|\tau_\sigma(Y)-\delta\bigr|\  +\ \bigl|\tau_\sigma(\sigma(Y))-\delta\bigr| = 0.
\]
It follows that $\tau_\sigma(Y)$ and $\tau_\sigma(\sigma(Y))$ are both equal to $\delta$, and in particular are equal to each other, as desired.

It remains to prove that $F$ is well defined.  Similar to the proof of Proposition~\ref{prop:metricmumford}, the intertwining with $\phi$ gives that 
\[
F\bigl(\pi(X)\bigr) = \min_{g\in L_2}\left\{ \bigl|d\bigl(X,g^{-1}\phi(g)\cdot\sigma(X)\bigr) - \delta\bigr| + \bigl|d\bigl(\sigma(X),\phi(g)^{-1}\phi^2(g)\cdot\sigma^2(X)\bigr) - \delta\bigr| \right\}.
\]
Again, since the action of $P$ on $T$ is properly discontinuous, the same is true for $L_2$.  Thus, the minimum exists and $F$ is well defined.
\end{proof}

\p{Forward translations along rays} The next proposition is a version of one of the steps of the Bers proof of the Nielsen--Thurston classification \cite[Section 13.6.4, Step 1]{primer}.  Here we generalize to the case where $\tau_\sigma$ is non-expanding.  The proof is almost unchanged.  We begin by defining forward translation along a ray.

Let $\gamma$ be a ray in a metric space $T$, and say that $\gamma$ has a unit speed parameterization as $\gamma : [0,\infty) \to T$.  For any interval $J \subset [0,\infty)$ we have a (possibly infinite) segment $\gamma|J$ of $\gamma$.  The forward translation of $\gamma|J$ along $\gamma$ by $d$ is the segment $\gamma : J \to \gamma$ given by 
\[
\gamma(t) \mapsto \gamma(t+d).
\]
This map is an isometric embedding of $\gamma$ into itself.

\begin{proposition}
\label{prop:metricray}
Let $(T,P,\phi,\sigma)$ be a synthetic Nielsen--Thurston package.  Suppose $\tau_\sigma$ is positive and that $X \in T$ realizes $\tau_\sigma$.  Let $\gamma$ be the geodesic ray from $X$ through $\sigma(X)$.   Then $\sigma|\gamma$ is the forward translation of $\gamma$ by $\tau_\sigma$.  In particular, $\sigma$ is not weakly contracting.
\end{proposition}

\begin{proof}

Let $Y$ be a point on $\gamma$ between $X$ and $\sigma(X)$.  Using the triangle inequality twice and the assumption that $\sigma$ is non-expanding, we have
\begin{align*}
d(Y,\sigma(Y)) &\leq d(Y,\sigma(X)) + d(\sigma(X),\sigma(Y)) \\
&\leq d(Y,\sigma(X)) + d(X,Y) \\
& = d(X,\sigma(X)) \\
&= \tau_\sigma.
\end{align*}
By the definition of $\tau_\sigma$ as an infimum, each of the above inequalities is an equality.  By the first (in)equality and the assumption that $T$ is uniquely geodesic, it must be that $\sigma(Y)$ lies on $\gamma$.  By the second (in)equality, $\sigma$ preserves the distance between $X$ and $Y$.  Combining the last two statements and the fact that $Y$ was arbitrary, we find that the restriction of $\sigma$ to the initial segment of $\gamma$ from $X$ to $\sigma(X)$ is forward translation along $\gamma$ by $\tau_\sigma$.  Inductively, we see that the restriction of $\sigma$ to the segment of $\gamma$ from $\sigma^k(X)$ to $\sigma^{k+1}(X)$ is forward translation along $\gamma$ by $\tau_\sigma$, whence the proposition.
\end{proof}


\section{Proof of the \"Ubertheorem: Non-exceptional cases}
\label{sec:proof}

In this section we combine the results of the previous three sections to prove the Nielsen--Thurston \"Ubertheorem in the non-exceptional cases.  In preparation, we present some of the requisite terminology and state and prove a series of three lemmas.

\p{Modulus} For $r > 1$ the modulus of the standard annulus $1 < |z| < r$ is $\ln r/2\pi$.  The modulus of an an arbitrary annulus (annular domain) is the modulus of the unique standard annulus to which it is biholomorphic.  We note that the standard annulus is conformally equivalent to a Euclidean cylinder of height $\ln r$ and circumference $2\pi$. 

For $X \in \Teich(\Sigma)$ and $A \subseteq \Sigma$ and embedded annulus we denote by $\mu_x(A)$ the modulus of $A$.  Similarly, for $\gamma$ a simple closed curve in $\Sigma$ we denote by $\mu_X(\gamma)$ the supremum of $\mu_X(A)$ over all embedded annuli $A$ in $\Sigma$ homotopic to $\gamma$.   We denote by $\mu(X)$ the supremum of $\mu_X(\gamma)$ as $\gamma$ ranges over all simple closed curves in $\Sigma$.

\p{Covering modulus} We require another version of modulus.  Let $\gamma$ be an essential closed curve in a Riemann surface $X$.  There is an annular cover $\tilde X_\gamma \to X$ corresponding to $\gamma$, which is unique up to biholomorphism.  We define the covering modulus of $\gamma$ to be
\[
\tilde \mu_X(\gamma) = \mu(\tilde X_\gamma).
\]
It is a fact that $\tilde \mu_X(\gamma)$ is $\pi/\ell_X(\gamma)$, where $\ell_X(\gamma)$ is the length of the geodesic in the free homotopy class of $\gamma$, with respect to the hyperbolic metric associated to $X$. 

\p{The Margulis number} The Margulis number $\epsilon$ is a real number with the properties that (1) any closed curve  $\gamma$ with covering modulus $\tilde \mu_X(\gamma) > \epsilon$ is a multiple of a simple closed curve, and (2) if $\gamma_1$ and $\gamma_2$ are simple closed curves with $\mu_X(\gamma_i) \geq \epsilon$, then there are disjoint annuli homotopic to $\gamma_1$ and $\gamma_2$, respectively, each of modulus $\epsilon' = \mu_X(\gamma_i) -1$; see \cite[Lemma 13.6]{primer}.  The second fact, sometimes called the collar lemma, implies that if two simple closed curves in $\Sigma$ have modulus greater than or equal to $\epsilon$ then they are homotopic to disjoint curves.

Let $\xi(\Sigma)$ denote the maximum number of pairwise disjoint, pairwise non-homotopic, simple closed curves in $\Sigma$. This is an upper bound for the number of homotopy classes of simple closed curves $\gamma$ with $\mu_X(\gamma) > \epsilon$.

\p{Modulus-degree inequality} Let $f \colon X' \to X$ be a (holomorphic) covering map of Riemann surfaces, and let $\gamma'$ be a component of $f^{-1}(\gamma)$.  We denote by $\deg f|\gamma'$ the degree of the restriction of $f$ to $\gamma'$.   Then
\[
\mu_{X'}(\gamma') \leq \frac{\mu_X(\gamma)+1}{\deg f|\gamma'}.
\]
This fact, which we refer to as the modulus-degree inequality, follows from two other facts: (1)~the covering modulus multiplies by exactly $\deg f|\gamma'$ under the cover, and (2) the fact that 
\[
\tilde \mu_X(\gamma)-1 \leq \mu_X(\gamma) \leq \tilde \mu_X(\gamma).
\]
The right-hand inequality here is immediate, since an annulus in $X$ lifts to an annulus in $\tilde X_\gamma$.  The left-hand inequality follows from the quantitative version of the collar lemma given above.  The left-hand inequality also follows from Maskit's comparisons between extremal length and modulus \cite[Propositions 1 and 2]{maskit}.

\p{The Gr\"otzch inequality} The next ingredient  is a version of the classical Gr\"otzch inequality, adapted from the case of rectangles to the case of annuli; see \cite[Theorem 11.10]{primer}.  It states that given $X,Y \in \Teich(\Sigma)$, a $K$-quasiconformal map $h : X \to Y$, and a simple closed curve $\gamma$ in $\Sigma$ we have
\[
\frac{1}{K}\mu_X(\gamma) \leq \mu_Y(h(\gamma)) \leq K \mu_X(\gamma).
\]
Applying this fact to the Teichm\"uller map $h :X \to Y$ we obtain
\[
\frac{1}{e^{2d(X,Y)}} \mu_X(\gamma) \leq \mu_Y(\gamma) \leq e^{2d(X,Y)} \mu_X(\gamma).
\]

\p{Finding stable multicurves}  If $X\in\Teich(\Sigma)$ and $\Gamma$ is a multicurve in $\Sigma$, let $\mu_X(\Gamma)$ denote the vector of moduli of the components of~$\Gamma$ (we emphasize that each component is the modulus of a single curve).   Also, for a dynamical branched cover $f : \Sigma \to \Sigma$ and $\Gamma$ a multicurve in $\Sigma$, we define the \emph{full preimage} of $\Gamma$ to be the set of all homotopy classes of simple closed curves in $\Sigma$ that map to components of $\Gamma$ under a power of $f$.  The following lemma is essentially the same as Proposition~8.1(a) in Douady--Hubbard.  

\begin{lemma}\label{lem:InvariantMulticurves}
Let $f \colon \Sigma\to \Sigma$ be a dynamical branched cover, and let $D > 0$.  There exists an $N>0$, depending only on $\Sigma$, $\deg f$, and $D$ with the following property: for any multicurve $\Gamma$ in~$\Sigma$ and any $X \in \Teich(\Sigma)$ with
\[
\mu_X(\Gamma) > (N,\ldots,N) \quad \text{and} \quad 
   \tau_{\sigma_f}(X) \leq D,
\]
the full preimage of $\Gamma$ is an $f$-stable multicurve.
\end{lemma}

\begin{proof}

Let $K=e^{2D}$, and let $N=(Kd)^{\xi(\Sigma)}\epsilon$, where $d$ is the degree of~$f$.  For each $j \geq 0$ let $\Gamma_j$ be the collection of all homotopy classes of essential curves in $f^{-i}(\Gamma)$ for $0\leq i\leq j$. 

We claim that for $0 \leq j \leq \xi(\Sigma)$ the collection $\Gamma_j$ is a multicurve.  By the properties of the Margulis constant $\epsilon$, it suffices to show that each component of $\Gamma_j$ has modulus bounded below by $\epsilon$.  We now prove this.  Since $\sigma_f$ is non-expanding and $\tau_{\sigma_f}(X)\leq D$, we have $\tau_{\sigma_f^i}(X) \leq iD$ for all $i\geq 0$, so each of the associated Teichm\"uller maps $X\to \sigma_f^i(X)$ is \mbox{$K^i$-quasiconformal}. 
Let $\gamma'$ be a component of $\Gamma_j$; say $\gamma'$ is a component of $f^{-i}(\Gamma)$.  By the Gr\"otzch inequality, we have
\[
\mu_{X}(\gamma') \geq  \frac{\mu_{\sigma_f^i(X)}(\gamma')}{K^i} \geq \frac{\mu_X(\gamma)}{K^id^i} \geq \frac{N}{K^id^i} \geq \frac{N}{K^{\xi(\Sigma)}d^{\xi(\Sigma)}} = \epsilon
\]
(for the second inequality, we use the fact that if we restrict a degree $d^i$ cover to a cover of annuli, then the latter has degree at most $d^i$).  Since $\gamma'$ was arbitrary, the claim follows.  

We next claim that some $\Gamma_j$ is $f$-stable.  Indeed, we have inclusions $\Gamma_0 \subseteq \Gamma_1 \subseteq \cdots \subseteq \Gamma_{\xi(S)}$.  Since $\Gamma_{\xi(S)}$ is a multicurve, we know that $|\Gamma_{\xi(S)}|\leq \xi(S)$, so there exists a $j< \xi(S)$ such that $\Gamma_j=\Gamma_{j+1}$, which implies that $\Gamma_j$ is an $f$-stable multicurve, as desired.
\end{proof}

\p{Uniform contraction} The following lemma is a basic linear algebra fact.  We will use it in the proof of the \"Ubertheorem to show that if a stable multicurve is not a strong reduction system, then under pullback (by a suitable power) the moduli of the curves fails to increase.  

For a matrix $A$, let $\|A\|$ denote the operator norm of a matrix $A$ with respect to the sup norm on $\R^n$.  We also denote by $\| \vec v \|$ the sup norm of $\vec v \in \R^n$.   We denote by $\rho(A)$ the spectral radius of $A$.  

\begin{lemma}
\label{lem:m}
There exists a number $p = p(\Sigma,d)$ with the following property.  If $f : \Sigma \to \Sigma$ is a dynamical branched cover of degree $d$ with $f$-stable multicurve $\Gamma$ and associated transition matrix $A$ then
\[
\rho(A) < 1 \quad\Rightarrow\quad \|A^p\| < \frac{1}{2}.
\]
\end{lemma}

Before giving the proof of Lemma~\ref{lem:m}, we remark that for a matrix $A$, the condition that $\rho(A)<1$ does not in general put any upper bound on $\|A\|$.

\begin{proof}[Proof of Lemma~\ref{lem:m}]

It follows from Jordan canonical form that if $\rho(A)<1$ then $\|A^n\|\to 0$ as $n\to \infty$.  In particular, there exists an $N_A$ such that $\|A^n\|<1/2$ for all $n\geq N_A$.

For a given degree $d$ and a given $\Sigma$ there are only finitely many possible transition matrices, and in particular finitely many for which $\rho(A)<1$. Taking the maximum of all corresponding $N_A$ yields the desired exponent~$p$.
\end{proof}

\p{The transition matrix versus the pullback map}  Let $f : \Sigma \to \Sigma$ be a dynamical branched cover.  If $\Gamma$ is an $f$-stable multicurve, there is an associated transition matrix $M$.   The $ij$-th entry is
\[
m_{ij} = \sum_\delta \frac{1}{\deg f|\delta}
\]
where $\delta$ is a component of $f^{-1}(\gamma_j)$ homotopic in $\Sigma$ to $\gamma_i$.  Here, $\deg f|\delta$ is the degree of the map $f|\delta : \delta \to \gamma_j$, thought of as a map $S^1 \to S^1$.  

For an $f$-stable multicurve $\Gamma$, the next lemma bounds (under certain conditions) the effect of $\sigma_f$ on $\mu_X(\Gamma)$ in terms of the associated transition matrix.  This statement incorporates Theorem 7.1, Proposition 8.1(b), and Proposition 8.2 in Douady--Hubbard as well as part of their proof of Proposition 8.2.

For the proof we use the notion of a latitude in an annulus.  By definition, an annulus $A$ in a Riemann surface is a subset that is biholomorphic to a standard annulus $A_r$ given by $1 < |z| < r$.  A latitude in $A_r$ is any circle centered at 0, and a latitude in $A$ is any corresponding circle in $A$ (under a biholomorphism).  A biholomorphism of $A_r$ preserves latitudes, and so the latitudes in $A$ form a well-defined foliation of $A$.

\begin{lemma}
\label{lemma:fudge}
Fix $d \geq 2$ and $\Sigma$ a marked surface.  Let $b = (d|P|+1)(\epsilon+2)$.  If $f \colon \Sigma \to \Sigma$ is a dynamical branched cover of degree $d$ with stable multicurve $\Gamma$ and associated transition matrix $M$, and for some $X \in \Teich(\Sigma)$ the multicurve $\Gamma$ includes all simple closed curves $\gamma$ with $\mu_X(\gamma) > \epsilon$, then 
\[
\mu_{\sigma_f(X)}(\Gamma) \leq M \mu_X(\Gamma) + (b,\dots,b).
\]
\end{lemma}

\begin{proof}

The given inequality is a vector inequality, which must hold separately for each component. Specifically, for each curve $\gamma$ of $\Gamma$, we must prove that
\[
\mu_{\sigma_f(X)}(\gamma)\leq \sum_{\delta \in \Delta_\gamma} \frac{\mu_X(\gamma)}{\deg f|\delta}+ b
\]
where $\Delta_\gamma$ is the set of all components of $\Delta = f^{-1}(\Gamma)$ that are homotopic to~$\gamma$ in $\Sigma$. Let $A$ be an annulus in $\sigma_f(X)$ homotopic to $\gamma$.  It suffices to prove that
\[
\mu_{\sigma_f(X)}(A) \leq \sum_{\delta\in \Delta_\gamma} \frac{\mu_X(\gamma)}{\deg f|\delta}+ b.
\]
We now set about proving this inequality.  

Let $\tilde X$ be the marked Riemann surface obtained from $\sigma_f(X)$ by adding additional marked points: the set of marked points $\tilde P$ is the full $f$-preimage of the marked points in $X$.  We have $|\tilde P| \leq d|P|$, and hence the maximal number of parallel, disjoint curves in $\tilde X$ is bounded above by $d|P|+1$.  In particular, $|\Delta_\gamma| \leq d|P|+1$. 

Decompose $A$ into sub-annuli $A_1,\ldots,A_n$ by cutting it along all latitudes that pass through marked points  of $\tilde X$.  For each $i$, let $\alpha_i$ be a latitude of $A_i$; we have $\mu_{\tilde X}(\alpha_i)\geq \mu_{\tilde X}(A_i)$. The curves $\alpha_1,\ldots,\alpha_n$ are pairwise non-isotopic in $\tilde{X}$---this is obvious except for the bottom curve $\alpha_1$ and the top curve $\alpha_n$, but if $n\geq 2$ then $\tilde{P}$ and hence $P$ must be nonempty, in which case any point of $P$ separates $\alpha_1$ from $\alpha_n$. As in the last paragraph, it follows that $n \leq d|P|+1$. Since we decomposed $A$ along latitudes, we have
\[
\mu_{\sigma_f(X)}(A) = \sum_{i=1}^n \mu_{\tilde X}(A_i).
\]
Set
\[
\A^\leq = \{A_i \mid \mu_{\tilde X}(A_i) \leq \epsilon+1 \} \qquad \text{and} \qquad \A^> = \{A_i \mid \mu_{\tilde X}(A_i) > \epsilon+1 \}
\]
where $\epsilon$ is the Margulis constant. We will prove two claims that provide upper bounds on the sum of moduli in $\A^\leq$ and $\A^>$ in turn, beginning with $\A^\leq$.

We first claim that
\[
\sum_{\A^\leq} \mu_{\tilde X}(A_i) \leq (d|P|+1)(\epsilon+1).
\]
This follows from the fact that $n \leq (d|P|+1)$, and the definition of $\A^\leq$.

We next claim that
\[
\sum_{\A^>} \mu_{\tilde X}(A_i) \leq \sum_{\delta\in \Delta_\gamma} \frac{\mu_X(\gamma)}{\deg f|\delta} + (d|P|+1).
\]
Consider an $A_i\in\A^>$.  Since each point of $f^{-1}(P)\subseteq \tilde X$ is marked, the image of $\alpha_i$ under $f$ is a curve $\gamma_i$ in~$X$.  This curve satisfies
\[
\tilde{\mu}_X(\gamma_i) =  \tilde{\mu}_{\tilde X}(\alpha_i) \geq \mu_{\tilde X}(\alpha_i) \geq \mu_{\tilde X}(A_i)  > \epsilon+1,
\]
Here the first step uses the fact that the annular cover for $\gamma_i$ is the same as the annular cover for $\alpha_i$, the second step uses the fact that any annulus homotopic to $\alpha_i$ lifts to the annular cover, the third step uses the fact that $A_i$ is an annulus homotopic to $\alpha_i$, and the last step uses the definition of~$\A^>$.

Since $\tilde{\mu}_X(\gamma_i) > \epsilon+1> \epsilon$, we have that $\gamma_i$ is homotopic to a multiple of a simple closed curve for each $A_i\in\A^>$, and $\mu_X(\gamma_i)> \epsilon$ by the collar lemma.  By hypothesis, it follows that $\gamma_i$ is homotopic to a multiple of a component of $\Gamma$.  Then $\alpha_i$ must be homotopic to a multiple of some curve $\delta_i\in\Delta_\gamma$, and since $\alpha_i$ is simple it must be homotopic to~$\delta_i$.  By the modulus-degree inequality, we have
\[
\mu_{\tilde X}(A_i) \leq 
\mu_{\tilde X}(\alpha_i) = \mu_{\tilde X}(\delta_i) \leq \frac{\mu_X(\gamma)+1}{\deg f|\delta_i}.
\]
Since the curves $\alpha_i$ are pairwise non-isotopic in $\tilde{X}$, the $\delta_i$'s are all distinct, so
\[
\sum_{\A^>} \mu_{\tilde X}(A_i) \leq \sum_{\A^>} \frac{\mu_X(\gamma)+1}{\deg f|\delta_i} \leq \sum_{\delta\in\Delta_\gamma} \frac{\mu_X(\gamma)+1}{\deg f|\delta} \leq \sum_{\delta\in\Delta_\gamma} \frac{\mu_X(\gamma)}{\deg f|\delta} + (d|P|+1).
\]
The first inequality is as above, the second inequality comes from the fact that the $\delta_i$'s are all distinct, and the third comes from two facts, namely, that $1/(\deg f|\delta) \leq 1$ and that $|\Delta_\gamma| \leq d|P|+1$.  This completes the proof of the claim.

We may now complete the proof of the lemma.  We have
\begin{align*}
\mu_{\sigma_f(X)}(A) 
= \sum_{i=1}^n \mu_{\tilde X}(A_i) 
= \sum_{\A^\leq} \mu_{\tilde X}(A_i) + \sum_{\A^>} \mu_{\tilde X}(A_i) 
\leq \sum_{\delta\in \Delta_\gamma} \frac{\mu_X(\gamma)}{\deg f|\delta} + b
\end{align*}
The first equality was explained above.  The second equality is true since $\{A_i\}$ is equal to the disjoint union $\A^\leq \cup \A^>$.  The last inequality is the combination of the two claims and the definition of $b$. 
\end{proof}

\p{Mapping class groups and virtual endomorphisms} Let $\Sigma = (S,P)$.  The pure mapping class group $\PMod(\Sigma)$ is the group of homotopy classes of homeomorphisms of $\Sigma$, where homeomorphisms and homotopies are required to fix $P$ pointwise.  

Let $f : \Sigma \to \Sigma$ be a dynamical branched cover.  There is an associated virtual endomorphism
\[
\phi : \PMod(\Sigma) \dasharrow \PMod(\Sigma)
\]
defined by lifting (homotopy classes of) homeomorphisms through $f$.  It follows from the usual lifting criterion in algebraic topology and the fact that the degree of $f$ is finite that the domain of $\phi$ has finite index in $\PMod(\Sigma)$.  Since isotopies always lift through $f$, the map $\phi$ is well defined.

There is a natural action of $\PMod(\Sigma)$ on $\Teich(\Sigma)$ by pullback: given $h \in \PMod(\Sigma)$ and $X \in \Teich(\Sigma)$ we obtain $h \cdot X$ by pulling  back the complex structure given by a representative of $X$ through a representative of $h$.  It follows from the definitions that the pullback map $\sigma_f$ is intertwined with $\phi$.

\p{Mumford's compactness criterion} We refer to the quotient of $\Teich(\Sigma)$ by $\PMod(\Sigma)$ as moduli space (often moduli space refers to the quotient by a larger group, the full mapping class group). Mumford's compactness criterion states that if $X_i$ is a sequence in $\Teich(\Sigma)$ and if the images of the $X_i$ leave every compact set in moduli space then $\limsup \mu(X_i) \to \infty$.

\begin{proof}[Proof of the \"Ubertheorem: Non-exceptional cases]

As in the statement of the theorem, $f\colon \Sigma \to \Sigma$ is a dynamical branched cover where $\Sigma = (S,P)$.  Assume that $f$ is not exceptional.  Let $\phi\colon \PMod(\Sigma) \dasharrow \PMod(\Sigma)$ be the virtual endomorphism associated to $f$, and let $\sigma\colon \Teich(\Sigma) \to \Teich(\Sigma)$ denote the pullback map.  

It follows from Teichm\"uller's theorems that the space $\Teich(\Sigma)$ is uniquely geodesic and that all maximal geodesics are bi-infinite.  It is also known that the action of $\PMod(\Sigma)$  on $\Teich(\Sigma)$ is properly discontinuous \cite[Theorem 12.2]{primer}.  We already stated that $\sigma$ is intertwined with $\phi$.  By Proposition~\ref{prop:pullback}, the map $\sigma$ is non-expanding.  In other words, the collection
\[
(\Teich(\Sigma), \PMod(\Sigma), \phi, \sigma)
\]
is a (not-at-all synthetic) synthetic Nielsen--Thurston package.  

Following the Bers proof of the Nielsen--Thurston classification, we treat three cases in turn:
\begin{enumerate}
    \item $\tau_\sigma=0$ and is realized
    \item $\tau_\sigma$ is not realized
    \item $\tau_\sigma > 0$ and is realized
\end{enumerate}
We will show in the three cases that $f$ is holomorphic, strongly reducible, and pseudo-Anosov, respectively.  

\bigskip

\noindent \emph{Case 1.} In this case it follows from the definitions that $f$ preserves a complex structure on $\Sigma$, which implies that $f$ has a holomorphic representative.

\bigskip

\noindent \emph{Case 2, $d=1$.} Let $D = \tau_\sigma + 1$, and let $N$ be the resulting constant from Lemma~\ref{lem:InvariantMulticurves}.  Let $X_i$ be a sequence of points in $\Teich(\Sigma)$ with $\tau_\sigma(X_i) \to \tau_\sigma$.  By Proposition~\ref{prop:metricmumford}, the (images of the) $X_i$ leave every compact subset of moduli space.  By Mumford's compactness criterion, we may choose a $k$ so that $\mu(X_k) > N$.  In particular there is a simple closed curve $\gamma$ in $\Sigma$ with $\mu_{X_k}(\gamma) > N$.   By Lemma~\ref{lem:InvariantMulticurves}, the full preimage of $\gamma$ is a stable multicurve $\Gamma$.  This $\Gamma$ is a reduction system and hence a strong reduction system.

\bigskip

\noindent \emph{Case 2, $d>1$.}  By Proposition~\ref{prop:pullback}, some iterate of $\sigma$ is weakly contracting.  Applying Proposition~\ref{prop:metricorbit} to this iterate, we conclude that (the image of) every orbit leaves every compact subset of moduli space.  Fix one such orbit $Y_i$.  Again by Mumford's compactness criterion the $\mu(Y_i)$ tend to infinity.  

We now introduce several constants.  Let $p = p(\Sigma,d)$ the the constant obtained from Lemma~\ref{lem:m}.  Since $\sigma$ is non-expanding, there exists a $D>0$ so that $\tau_\sigma(Y_i) \leq D$ for all~$i$, namely, $D=\tau_\sigma(Y_0)$. For this $D$, let $N=N(\Sigma,d,D)$ be the constant from Lemma~\ref{lem:InvariantMulticurves}.  

Next, let $b=b(\Sigma,d)$ be the constant from Lemma~\ref{lemma:fudge}, and let
\[
r = \max_M \big\| M^{p-1} + \cdots + M\big\|   \bigl\|(b,\ldots,b)\big\|
\]
where the maximum is taken over all transition matrices $M$ for dynamical branched covers of degree $d$ over $\Sigma$ (there are finitely many such matrices).  Finally, let
\[
C = \max \{ N, 2r , \epsilon \}.
\]
Since $\limsup \mu(Y_i) = \infty$, there exists a smallest $n$ with $\mu(Y_n)>C$.  Increasing $C$ if necessary, we may assume that $n\geq p$.  Let $\gamma$ be a simple closed curve in $\Sigma$ so that $\mu_{Y_n}(\gamma) > C$.  Then $\mu_{Y_n}(\gamma)> N$, so  Lemma~\ref{lem:InvariantMulticurves} tells us that the full $f$-preimage $\Gamma$ of $\gamma$ is an $f$-stable multicurve.

Suppose for the sake of contradiction that $\Gamma$ is not the multicurve underlying some strong reduction system for $f$, that is, the transition matrix $M$ for $\Gamma$ has $\rho(M)<1$.  By Lemma~\ref{lem:m} we have $\|M^p\| \leq 1/2$.  We thus have
\begin{align*}
\mu_{Y_n}(\gamma) \leq \| \mu_{Y_n}(\Gamma) \| &\leq \bigl\| M^p\, \mu_{Y_{n-m}}(\Gamma) + (M^{p-1} + \cdots + M)(b,\ldots,b)
\bigr\|  \\ 
 &\leq \|M^p\|\,\|\mu_{Y_{n-p}}(\Gamma) \| +     \big\| M^{p-1} + \cdots + M\big\|   \bigl\|(b,\ldots,b)
\bigr\|  \\ 
&< \frac{1}{2}C + r \leq \frac{1}{2}C + \frac{1}{2}C = C.
\end{align*}
In order, we used the definition of the sup norm, Lemma~\ref{lemma:fudge} (iteratively), the triangle inequality and the definition of the operator norm, Lemma~\ref{lem:m} and the choices of $n$ and $r$, the choice of $C$, and basic algebra.  The resulting inequality $\mu_{Y_n}(\gamma) \leq C$ contradicts the earlier assumption that $\mu_{Y_n}(\gamma)> C$, and we are done.

\bigskip

\noindent \emph{Case 3.} Let $X \in \Teich(\Sigma)$ be a point with $\tau_\sigma(X)=\tau_\sigma$. Let $\gamma$ be the unique geodesic ray passing through $X$ and $\sigma(X)$.  Since $\tau_\sigma>0$ by assumption, Proposition~\ref{prop:metricray} implies the restriction of $\sigma$ to $\gamma$ is forward translation by~$\tau_\sigma$.    In particular, $\sigma^2(X)$ lies on $\gamma$ and $d(X,\sigma^2(X))$ is twice $d(X,\sigma(X))$.

The ray $\gamma$ is determined by an ordered pair of measured foliations $(\F^+,\F^-)$ on $\Sigma$, each well defined up to scaling and isotopy.  The Teichm\"uller map $h : X \to \sigma(X)$ has $(\F^+,\F^-)$ as its associated foliations.  

As in Section~\ref{sec:pullback} there is a commutative diagram 
\[
 \begin{tikzcd}
  (\sigma(X),f^*(\F^+,\F^-)) \arrow[r,"h^f"] \arrow[d,"f"] & (\sigma^2(Y),f^*(\lambda\,\F^+,\tfrac{1}{\lambda}\,\F^-)) \arrow[d,"f"] \\
 (X,(\F^+,\F^-)) \arrow[r,"h"]  & (\sigma(X),(\lambda\,\F^+,\tfrac{1}{\lambda}\,\F^-)),
 \end{tikzcd}
\]
where $h^f$ is a pseudo-Teichm\"uller map with the same dilatation as $h$ and where $\lambda = e^{\tau_\sigma}$.  Since $d(X,\sigma(X))=d(\sigma(X),\sigma^2(X))$, it follows that $h^f$ is in fact a Teichm\"uller map.  

We claim that the top-left and bottom-right corners of the diagram are scalar multiples.  More precisely, we claim
\[
f^*(\F^+,\F^-) =  ((\sqrt{d}\lambda)\,\F^+,(\sqrt{d}/\lambda)\,\F^-),
\]
where $d=\deg(f)$.  This claim gives that $f$ is pseudo-Anosov, and so it remains to prove the claim.  (One is tempted to worry about the fact that $X \neq \sigma(X)$, but if we forget the complex structures, we can replace both $X$ and $\sigma(X)$ in the claim with $\Sigma$, making it clear how the claim implies that $f$ is pseudo-Anosov.)

Firstly, the underlying (unmeasured) foliations must be equal, for if not, the composition $h^f \circ h$ would have dilatation less than $\lambda^2$ and hence $d(X,\sigma^2(X))$ would be strictly less than $2d(X,\sigma(X))$, a contradiction.  As for the measures, the Euclidean areas of the pairs of foliations on the bottom row are equal, and pulling back by $f$ multiplies area by $d$, and so the claim follows.    

\bigskip

\noindent \emph{Exclusivity.} We now prove the exclusivity statement in the non-exceptional case.  As discussed in the introduction---and proved in the appendix---a strong reduction system is an obstruction to holomorphicity when $d > 1$.  This implies that cases 1 and 2 are exclusive when $d>1$.  We would now like to show that cases 2 and 3 are exclusive.  To this end, we first point out that in the above argument for Case 3, Proposition~\ref{prop:metricray} further implies that $\sigma$ is not weakly contracting.  Since we are in the non-exceptional case, Proposition~\ref{prop:pullback} then implies $\deg f = 1$, that is, $f$ is an element of the mapping class group of $\Sigma$.  Therefore, the exclusivity of Cases 2 and 3 follows as in the Nielsen--Thurston classification theorem (a pseudo-Anosov mapping class stretches the lengths of all curves exponentially, but a reducible mapping class does not \cite[Theorem 14.23]{primer}).  

\bigskip

\noindent \emph{Uniqueness.}  Finally, we prove the uniqueness statements of the theorem.  If $f$ is non-exceptional with $\deg f > 1$ then it follows from Proposition~\ref{prop:pullback} that $\sigma_f$ has an iterate that is weakly contracting.  In particular, $\sigma_f$ has at most one fixed point, and so there is at most one complex structure for which $f$ is holomorphic.  The other uniqueness statement is the same as in the case of mapping class groups, since (as above) all non-exceptional pseudo-Anosov maps have degree 1.  See \cite[Corollary 12.4]{FLP} for the argument.  The idea is that, under iteration, a pseudo-Anosov map acts with source-sink dynamics on the space of projective measured foliations.  %
\end{proof}

We record here two statements that were established in the course of the proof of the \"Ubertheorem in the non-exceptional cases.  These statements will be applied in the proof for the exceptional cases.

\begin{proposition}\label{case3}
Let $f : \Sigma \to \Sigma$ be a dynamical branched cover.  Suppose that the pullback map $\sigma_f$ has an orbit whose image in moduli space leaves every compact set.  Then $f$ is strongly reducible.
\end{proposition}

\begin{proposition}\label{case2}
Let $f : \Sigma \to \Sigma$ be a dynamical branched cover.  Suppose that the pullback map $\sigma_f$ preserves a geodesic ray in $\Teich(\Sigma)$ and acts by forward translation on that ray.  Then $f$ is pseudo-Anosov.
\end{proposition}

Even though our proof of Case 3 in the non-exceptional case reduces to the case of $\deg f = 1$, we gave the argument for arbitrary degree precisely so that we could give Proposition~\ref{case2}.


\section{Proof of the \"Ubertheorem: Exceptional cases}
\label{sec:torus}

In this section we prove the \"Ubertheorem in the remaining cases, the exceptional cases.  As above, these are the cases where $\deg f > 1$ and $f$ either a torus map or a sphere map obtained from a torus map through the hyperelliptic involution.  

The proof uses many of the tools developed in  Section~\ref{sec:proof}.  The main obstacle is that  Proposition~\ref{prop:stable} gives no information in the exceptional cases, and hence Proposition~\ref{prop:pullback} does not hold (as we will see, there are indeed cases where the pullback map has no iterate which is a weak contraction, namely, the cases of affine exceptional maps).  We will instead take advantage of a product structure on Teichm\"uller space that is special to the exceptional cases.  (In the case of an unmarked exceptional map, the product structure is trivial, and so these cases could be equally well have been addressed in Section~\ref{sec:proof}.)

The paper by Douady--Hubbard gives a detailed account of the dynamical branched covers with Euclidean orbifold, including a catalogue of all such maps \cite[Section 9]{DH}.  

\p{Exceptional surfaces and maps} In order to give proofs that work simultaneously for the torus and the sphere, we will slightly alter our notation for a marked surface.  Specifically, in this section, a marked surface $\Sigma$ is a pair $(S,P)$ where $S = (S_0,P_0)$ itself is a surface with marked points in the usual sense (so $S_0$ is a closed surface) and $P \subseteq S_0 \setminus P_0$. The relevant marked surfaces $\Sigma$ for this section are $((T^2,\emptyset),P)$ and $((S^2,P_0),P)$ with $|P_0|=4$.

When we say that a dynamical branched cover $f : \Sigma \to \Sigma$ is exceptional, we will take $\Sigma$ to be $(S,P)$ where $S = (S_0,P_0)$ as above and $P_0$ is the post-critical set of $f$.  So in all cases $P$ is the set of marked points that are not post-critical.  

\p{Teichm\"uller maps in the exceptional cases} For the torus $T^2$, a Teichm\"uller map is the same thing as an orientation-preserving affine homeomorphism.  This follows from the same reasoning as in the resolution of Gr\"otzsch's problem about extremal maps between rectangles \cite[Theorem 11.10]{primer}.  As a consequence, we see that Teichm\"uller maps on $T^2$ are closed under composition.  

We can identify $\Teich(T^2)$ with $\Teich(S_{1,1})$, the Teichm\"uller space of the torus with one marked point (this is the space of complex structures on the torus, modulo pullback by diffeomorphisms that fix the marked point and are homotopic to the identity).  For the latter, the Teichm\"uller maps are exactly the orientation-preserving linear homeomorphisms and they are thus unique.  In what follows, when we refer to \emph{the} Teichm\"uller map between two points of $\Teich(T^2)$, we mean the linear one (here we are abusing the identification of $\Teich(T^2)$ with $\Teich(S_{1,1})$).

Every point in $\Teich(T^2)$ comes equipped with a holomorphic hyperelliptic involution.  The quotient is a Riemann surface that may be regarded as a sphere with four marked points.  Each marked point corresponds to a fixed point of the hyperelliptic involution, also called a Weierstrass point.  This correspondence gives a homeomorphic identification of $\Teich(T^2)$ with $\Teich(S_{0,4})$, the Teichm\"uller space of a sphere with four marked points.

The Teichm\"uller maps for $S_{0,4}$ are exactly the quotients under the hyperelliptic involution of the affine maps of $T^2$ preserving the set of four Weierstrass points.  By the same token, the above correspondence of $\Teich(T^2)$ with $\Teich(S_{0,4})$ is an isometry.  

\p{A product decomposition on Teichm\"uller space} Let $\Sigma = (S,P)$ be an exceptional marked surface.  Again, either $S$ is $(T^2,\emptyset)$, or it is $S=(S^2,P_0)$ with $|P_0|=4$, and in either case $P \cap P_0 = \emptyset$.  There is a forgetful map
\[
\pi_v : \Teich(\Sigma) \to \Teich(S)
\]
obtained by forgetting the set of marked points $P$.  Let $X_\square \in \Teich(S)$ be some basepoint for $\Teich(S)$ (for instance when $S = T^2$, we may take $X_\square$ to be the unit square torus where the generators for $\pi_1(T^2)$ have length 1).  We denote $\pi_v^{-1}(X_\square)$ by $\Teich(X_\square,P)$.

Having defined $\Teich(X_\square,P)$ we may define a map
\[
\nu : \Teich(S)\times \Teich(X_\square,|P|) \to \Teich(\Sigma).
\]
The formula for $\nu$ is
\[
\nu(X,Y) = (h_X)_*(Y)
\]
where $h_X : X_\square \to X$ is the Teichm\"uller map and  $(h_X)_*$ is the push forward of the complex structure $Y$. 
 The marked points in $\nu(X,Y)$ are defined to be the $h_X$-images of the marked points in $Y$.  
 
In what follows we will refer to a subset $\Teich(S)\times \{Y\}$ of $\Teich(S)\times \Teich(X_\square,|P|) \to \Teich(\Sigma)$ as a horizontal slice, and we will write it as $\Teich(S)\times Y$ for simplicity.  We have a similar definition and notation for vertical slices.

\begin{proposition}
\label{prop:torus}
Let $\Sigma = (S,P)$ be an exceptional marked surface and fix some $X_\square \in \Teich(S)$.
\begin{enumerate}[label={(\arabic*)},ref={\theproposition(\arabic*)}]
\item\label{torus:product} The map
\[
\nu : \Teich(S)\times \Teich(X_\square,|P|) \to \Teich(\Sigma)
\]
is a homeomorphism.
\item\label{horisom} The map $\nu$ restricts to an isometry on each horizontal slice $\Teich(S)\times Y$.
\item\label{torus:hor} Two points $Z_1,Z_2 \in \Teich(\Sigma)$ lie in the $\nu$-image of a slice $\Teich(S)\times Y$ if and only if the Teichm\"uller map between them has no 1-pronged singularities at points of $P$.
\item\label{verticalproj} The projection $\pi_v : \Teich(\Sigma) \to \Teich(S)$ is non-expanding.  Further $d(\pi_v(Z_1),\pi_v(Z_2)) = d(Z_1,Z_2)$ if and only if $Z_1$ and $Z_2$ lie in the same horizontal slice $\Teich(S) \times Y$.
\end{enumerate}
\end{proposition}

\begin{proof}

We begin with the first statement.  To prove it, we define an inverse map to $\nu$.  The inverse has two coordinate functions.  The first is the projection map $\pi_v$.  The second coordinate function is: 
\[
\rho(Z) = h_{X}^*(Z)
\]
where $X=\pi_v(Z)$ and $h_{X}^*$ is pullback by the Teichm\"uller map $h_{X} : X_\square \to X$.  The maps $\nu$, $\pi_v$, and $\rho$ are well defined and continuous by Teichm\"uller's theorems.  The maps $\nu$ and $\pi_v \times \rho$ are inverses of each other by definition, and so both are homeomorphisms, proving the first statement.

We proceed to the second statement.  Let $(X_1,Y)$ and $(X_2,Y)$ be two points of $\Teich(S)\times \Teich(X_\square,|P|)$, and let $Z_1$ and $Z_2$ be their $\nu$-images.  Let $h : X_1 \to X_2$ be the Teichm\"uller map.  Since $\nu$ is defined in terms of Teichm\"uller maps from $X_\square$ and since Teichm\"uller maps of exceptional surfaces are closed under composition, it follows that $h$ may be regarded as the Teichm\"uller map $Z_1 \to Z_2$.  Since we have Teichm\"uller maps $X_1 \to X_2$ and $Z_1 \to Z_2$ with the same stretch factor (in fact it is the same underlying map), the second statement follows.

The third statement follows from the previous paragraph.  Indeed, if two points lie in the $\nu$-image of a horizontal slice, then we have from the previous paragraph a Teichm\"uller map with the desired properties.  For the other direction, suppose $h: Z_1 \to Z_2$ is a Teichm\"uller map where $Z_i = \nu(X_i,Y_i)$ and suppose $h$ has no singularities at the points of $P$.  We would like to show $Y_1 = Y_2$.  We may regard $h$ as a Teichm\"uller map $X_1 \to X_2$.  If $h_i : X_\square \to X_i$ is the Teichm\"uller map for each $i$ then $h \circ h_1 = h_2$.  Since $Y_i = h_i^*(Z_i)$, we have
\[
Y_2 = h_2^*(Z_2) = (h \circ h_1)^*(Z_2) = h_1^* h^*(Z_2) = h_1^*(Z_1) = Y_1,
\]
We now prove the fourth statement.  The projection $\pi_v$ is non-expanding because a Teichm\"uller map $h : Z_1 \to Z_2$ induces a pseudo-Teichm\"uller map $\bar h : \pi_v(Z_1) \to \pi_v(Z_2)$, as in Section~\ref{sec:pullback}.  The pseudo-Teichm\"uller map $\bar h$ is a Teichm\"uller map if and only if $h$ has no singularities at a point of $P$.  The fourth statement now follows from the third.
\end{proof}

Since Teichm\"uller maps between points in a horizontal slice are affine, the space $\Teich(X_\square,P)$---or indeed any of the vertical slices in the product decomposition in Proposition~\ref{prop:torus}---can be identified with the space of affine structures on $\Sigma$.  

\p{Pullback and the product decomposition} Given the product decomposition from Proposition~\ref{prop:torus}, our next goal is to elaborate on the interaction between the product structure and the pullback map.  The statement of the following proposition uses the following observation: an exceptional dynamical branched cover $f : (S,P) \to (S,P)$ induces a dynamical branched cover $\bar f : S \to S$.  In particular, there is an induced pullback map on $\Teich(S)$.

\begin{proposition}
\label{prop:toruspullback}
Let $\Sigma = (S,P)$ and let $f: \Sigma \to \Sigma$ be an exceptional dynamical branched cover of degree $d$.  
\begin{enumerate}[label={(\arabic*)},ref={\theproposition(\arabic*)}]
     \item\label{preserveproduct} The pullback map $\sigma_f$ preserves the product structure on $\Teich(\Sigma)$. 
    \item\label{horisom} If $\sigma_f$ preserves a horizontal slice $H$ of $\Teich(\Sigma)$ then $f$ is affine, $\sigma_f|H$ is an isometry, and $\sigma_f|H$ is conjugate under $\pi_v|H$ to the induced pullback map $\sigma_f^{hor}$ on $\Teich(S)$.
    \item\label{noslice} If $\sigma_f$ preserves no horizontal slice of $\Teich(\Sigma)$, then all $\sigma_f$-orbits are weakly contracting.
\end{enumerate}
\end{proposition}

For an exceptional $\Sigma = (S,P)$ we have that $\Teich(S)$ is isometric to $\H^2$ (up to scale).  And by Proposition~\ref{verticalproj} the restriction of $\pi_v$ to each horizontal slice of $\Teich(\Sigma)$ is an isometry to $\Teich(S)$.  Thus, Proposition~\ref{horisom}, implies that $\sigma_f$ is isometrically conjugate, through $\pi_v$, to an isometry of $\H^2$. 

\begin{proof}[Proof of Proposition~\ref{prop:toruspullback}]

We begin with the first statement.  It follows from the definitions that $\sigma_{\bar f} \circ \pi_v = \pi_v \circ \sigma_f$, and hence that $\sigma_f$ preserves the set of vertical slices of the product.

Now suppose that $Z_1$ and $Z_2$ lie in the same horizontal slice.  By Proposition~\ref{torus:hor} the Teichm\"uller map $h : Z_1 \to Z_2$ has no 1-pronged singularities at $P$.  Since the map $h$ is homotopic to the identity, it has a lift through $f$.  We denote this lift by $\tilde h$.  By the definition of the pullback, we have that $\tilde h$ maps $\sigma_f(Z_1)$ to $\sigma_f(Z_2)$, in the sense that $\tilde h^*(\sigma_f(Z_2))=\sigma_f(Z_1)$.  

By Proposition~\ref{torus:hor}, the first statement is a consequence of the following claim: the map $\tilde h$ is the Teichm\"uller map $\sigma_f(Z_1) \to \sigma_f(Z_2)$ and the singularities for the associated foliations all lie at $P_0$.  Since $\tilde h$ is the lift of $h$ through $f$, it is a pseudo-Teichm\"uller map whose foliations are the preimages of the foliations for $h$.  Since the 1-pronged singularities for the latter all lie at points of $P_0$, and since in both exceptional cases the preimage of $P_0$ is the union of $P_0$ with the set of critical points for $f$, it follows that the foliations for $\tilde h$ have 1-pronged singularities only at $P_0$ and that $\tilde h$ is a Teichm\"uller map, as desired.  

Suppose now that $\sigma_f$ preserves a horizontal slice $H$ of $\Teich(\Sigma)$.  From the equality $\sigma_{\bar f} \circ \pi_v = \pi_v \circ \sigma_f$ used above, we conclude that $\sigma_f|H$ is conjugate under $\pi_v|H$ to the induced pullback map $\sigma_f^{hor} \colon \Teich(S) \to \Teich(S)$, as in the second statement.

We next prove that if $\sigma_f$ preserves a horizontal slice, then $f$ is affine (as in the second statement). By the definition of the product structure on $\Teich(\Sigma)$, its horizontal slices correspond exactly to the (singular) affine structures on $\Sigma$.  Therefore, if $f$ preserves a horizontal slice, it preserves an affine structure, and hence is affine. 

The remaining two statements (really the third statement and the second conclusion of the second statement) will be consequences of the following claim: if $X$ and $Y$ are points of $\Teich(\Sigma)$, then $d(\sigma_f(X),\sigma_f(Y))$ is strictly less than $d(X,Y)$ if and only if $X$ and $Y$ lies in different horizontal slices.  Indeed, by Proposition~\ref{torus:hor}, $X$ and $Y$ lie in different horizontal slices if and only if the foliations for the Teichm\"uller map $h : X \to Y$ have a 1-pronged singularity at a point of $P$.  Since the points of $P$ are not post-critical (by definition), the latter is true if and only if the foliations for the lifted map $\tilde h : \sigma_f(X) \to \sigma_f(Y)$ have a 1-pronged singularity at a point of $f^{-1}(P)$.  Since $f^{-1}(P)$ is disjoint from $P_0$ (again using the fact that the points of $P$ are not post-critical), the claim now follows from a second application of Proposition~\ref{torus:hor}.

Suppose that $\sigma_f$ preserves a horizontal slice $H$.  By the claim and the fact that $\sigma_f$ is non-expanding (Proposition~\ref{nonexp}), it follows that $\sigma_f|H$ is an isometry.

Finally, if $\sigma_f$ preserves no horizontal slice then by the first statement it follows that for any $Z \in \Teich(\Sigma)$, the image $\sigma_f(Z)$ lies in a different horizontal slice of $\Teich(\Sigma)$.  Combining this with the claim completes the proof. 
\end{proof}

\begin{proof}[Proof of the \"Ubertheorem: Exceptional cases]

As in the statement of the theorem, $f\colon \Sigma \to \Sigma$ is an exceptional dynamical branched cover with degree $d > 1$.  In particular, we have that $\Sigma = (S,P)$ with either  $S=(T^2,\emptyset)$ or $S=(S^2,P_0)$ with $|P_0|=4$.   In either case, the marked points of $S$ are the post-critical points for $f$.  

By Lemma~\ref{preserveproduct}, $\sigma_f$ preserves the product structure on $\Teich(\Sigma)$.  We treat two cases, according to whether or not $\sigma_f$ preserves a horizontal slice of $\Teich(\Sigma)$.

If $\sigma_f$ preserves no horizontal slice then by Lemma~\ref{noslice}, each $\sigma_f$-orbit is weakly contracting.  By Proposition~\ref{prop:metricorbit}, each $\sigma_f$-orbit leaves every compact subset of moduli space.
Then by Proposition~\ref{case3}, the map $f$ strongly reducible.

Now suppose $\sigma_f$ does preserve a horizontal slice $H$.  By parts (2) and (3) of Proposition~\ref{prop:toruspullback}, the restriction $\sigma_f|H$ is isometrically conjugate to an isometry $\varphi$ of $\Teich(S) \cong \H^2$.  There are three possibilities for $\varphi$: it can be elliptic, loxodromic, or parabolic.

If $\varphi$ is elliptic then $\sigma_f|H$, hence $\sigma_f$, has a fixed point and $f$ is holomorphic.  And if $\varphi$ is loxodromic, then by Proposition~\ref{horisom} and Proposition~\ref{case2}, the map $f$ is pseudo-Anosov.  

In the remainder of the proof we deal with the case where $\varphi$ is parabolic.  In this case, the translation length of $\varphi$ is 0.  It then follows from Proposition~\ref{horisom} that the translation length $\tau_f$ is 0.  It also follows from Proposition~\ref{horisom} and Proposition~\ref{verticalproj} that this translation length is not realized by $f$ (translation distances in $\Teich(\Sigma)$ are no smaller than the corresponding translation distances in $H$).  

By Proposition~\ref{horisom}, the map $f$ is an affine torus map or a hyperelliptic quotient of an affine torus map. We first treat the case where $\Sigma$ is a torus and $f$ is affine.

Since $\varphi$ is parabolic, the linear map homotopic to $f$ must have a single repeated eigenvalue, namely $\sqrt{d}$.  We can change coordinates so that $f$ is of the form
\[
\left(\begin{array}{cc}
  \sqrt{d}   & \ast \\
    0 & \sqrt{d}
\end{array}\right)
\]
where $d = \deg(f)$.  It must be that $\sqrt{d}$ is a natural number.  The preimage under $f$ of any horizontal curve in $T^2$ is a collection of horizontal curves.  We will construct a strong reduction system consisting of horizontal curves.

Let $\Gamma = \{\gamma_1,\dots,\gamma_k\}$ be a maximal multicurve in $\Sigma$ consisting of horizontal curves.  The number of components $k$ is the same as the number of horizontal curves in $T^2$ that pass through a marked point of $\Sigma$ (although such curves are not permitted to be components of $\Gamma$, exactly because they pass through marked points).  We label each component $\gamma_i$ by its modulus (equivalently, the supremum of Euclidean widths of annuli in $\Sigma$ that have horizontal boundary curves and that contain the given $\gamma_i$).  These numbers are the vertical distances between marked points with distinct, but consecutive, coordinates in the vertical direction.

We claim that the resulting labeled multicurve, which we still call $\Gamma$, is a strong reduction system for $f$.  For each $i$, we may choose a closed annulus $A_i$ that has horizontal boundary, that has Euclidean width $\ell_i$, and that is homotopic in $\Sigma$ to $\gamma_i$.  If $\Sigma$ has marked points, then each $A_i$ has at least one marked point on each boundary component, and the union of all of the $A_i$ is $\Sigma$.  (If $\Sigma$ has no marked points, then $k=1$ and $A_1$ should be taken to be all of $T^2$.)  Each $f^{-1}(A_i)$ is a collection of $\sqrt{d}$ annuli, each with width $\ell_i/\sqrt{d}$ (the above matrix for $f$ stretches in the vertical direction by $\sqrt{d}$).  Since $f$ is a covering map, the union over $i$ of the $f^{-1}(A_i)$ is all of $\Sigma$, from which it follows that $\Gamma$ is a strong reduction system and so $f$ is strongly reducible, as desired. 

Suppose now that $\Sigma = ((S^2,P_0),P)$.  Since $f$ is exceptional and $\sigma_f$ preserves a horizontal slice of $\Teich(\Sigma)$, it follows from Proposition~\ref{horisom} that $f$ is affine.  Thus, $f$ lifts to an affine map $\tilde f$ of $T^2$.  What is more, $\tilde f$ can be regarded as an affine map of $\tilde \Sigma = (T^2,\tilde P)$, where $\tilde P$ is the preimage of $P$ under the hyperelliptic involution.  As above we obtain a strong reduction system $\tilde \Gamma$ in $\tilde \Sigma$, which we may assume is horizontal. By construction, $\tilde \Gamma$ is invariant under the hyperelliptic involution.  Hence it gives rise to a labeled multicurve $\Gamma$ in $\Sigma$.  Let $\HM(\Sigma)$ denote the set of labeled horizontal multicurves on $\Sigma$ and let $\SHM(\tilde \Sigma)$ denote the set of symmetric labeled horizontal multicurves on $\tilde \Sigma$ (we concentrate on horizontal curves to avoid curves in $\Sigma$ with connected preimage).  There is a commutative diagram
\[
 \begin{tikzcd}
  \SHM(\tilde \Sigma) \arrow[r,"\tilde f^\ast"] \arrow[d,"\cong"] & \SHM(\tilde \Sigma) \arrow[d,"\cong"] \\
 \HM(\Sigma) \arrow[r,"f^\ast"]  & \HM(\Sigma)
 \end{tikzcd}
\]
(where the horizontal maps are the natural pullback maps).  The symmetric, horizontal strong reduction system for $\tilde f$ thus gives a (horizontal) strong reduction system for $f$, as desired.

For exceptional maps, the only exclusivity statement is that types 1 and 3 are exclusive.  This follows by the same reasoning as in the non-exceptional case.  (In Appendix~\ref{sec:orb} we explain why the argument for exclusivity of types 1 and 2 only applies in the non-exceptional cases.)  The uniqueness statement for type 3 (pseudo-Anosov) maps follows from the same argument as in the non-exceptional case.
\end{proof}

We end by pointing out one consequence of the proof that is heretofore unmentioned: an exceptional dynamical branched cover of $\Sigma=(S,P)$ is affine if and only if it has no strong reduction system that is inessential in $S$.


\appendix

\section{Strong reduction systems and Thurston obstructions}
\label{sec:orb}

Our main goal in this appendix is to give a geometric characterization of the orbifold for a dynamical branched cover.  With this characterization, we accomplish two goals:
\begin{enumerate}
\item we give a direct proof that strong reduction systems are obstructions to holomorphicity for dynamical branched covers with hyperbolic orbifold,
\item we show that a dynamical branched cover of the sphere is exceptional if and only if its orbifold is the $(2,2,2,2)$-orbifold, and
\end{enumerate}
The first item explains why strong reduction systems are the ``obvious'' obstructions to holomorphicity for a non-exceptional dynamical branched cover.  The second justifies our characterization of exceptional maps in the introduction.  All of the material in this section was surely known to Thurston, although the authors are unable to find the arguments in the existing literature.  The argument in Theorem~4.1 of Douady--Hubbard is very similar to our argument for the first item.  Their proof concludes by considering the derivative of the pullback map on Teichm\"uller space, which in turn relies on their analogue of our Proposition~\ref{prop:stable}.  Our argument ends by simply considering the lifted map of the hyperbolic plane.

\p{Orbifolds for dynamical branched covers} For our purposes, a (2-dimensional) orbifold is a marked surface $(S,P)$ endowed with a labeling of $P$ by $\mathbb{N} \cup \{\infty\}$, that is, a function $\nu_P : P \to \mathbb{N} \cup \{\infty\}$.  If $\nu_P(p) > 1$ then we refer to $p$ as a cone point.  We will explain below the geometric meaning of an orbifold, which will allow us to use geometry to study dynamical branched covers.

A map $f : (S,P) \to (T,Q)$ is an \emph{orbifold cover} if it induces a branched covering map $S \to T$ and whenever we have $p \in P$, $q \in Q$, and $f(p)=q$, then
\[
(\deg f_p) \cdot \nu_p = \nu_q.
\]
Here, $\deg f_p$ is the local degree of $f$ at $p$.  

For two orbifolds $(S,P)$ and $(S',P')$ we write $(S',P') \sqsubseteq (S,P)$ if
 \begin{itemize}
 \item $S' \subseteq S$,
 \item $P' \subseteq P$, and
 \item for each $p \in P'$ we have $\nu_{P}(p) \mid \nu_{P'}(p)$.
\end{itemize}

A \emph{partial orbifold cover} from $(S,P)$ to $(T,Q)$ is an orbifold cover 
\[
(S',P') \to (T,Q)
\]
with $(S',P') \sqsubseteq (S,P)$.  And a \emph{partial self-orbifold cover} of an orbifold $(S,P)$ is a partial orbifold cover from $(S,P)$ to itself.  To our knowledge this definition has not appeared in the literature, although we strongly suspect it was known to Thurston.

A partial self cover of surfaces is a covering map $S' \to S$ where $S' \subseteq S$ (we sometimes require $S'$ to be open in $S$).  We can think of this as a special case of a partial self-orbifold cover, since a deleted point can be regarded as an orbifold point with label $\infty$.  

For a given dynamical branched cover $f : (S,P) \to (S,P)$, a basic problem is to understand all orbifold structures on $(S,P)$ so that $f$ induces a partial self-orbifold cover of $(S,P)$.  Specifically, this means that there is some $(S',P') \sqsubseteq (S,P)$ so that the induced map $f : (S',P') \to (S,P)$ is an orbifold cover.  Once we explain the geometric meaning of orbifolds below, we will be able to use the geometry of the orbifold to study $f$.  

Given $f : (S,P) \to (S,P)$, there is a minimal labeling of $P$ so that $f$ is a partial self-orbifold covering map.  The label at $p \in P$ is determined as follows.  For each $k$ and each critical point $c$ with $f^k(c)=p$, we compute the local degree of $f^k$ at $c$.  The label $\nu_p$ is the least common multiple of these local degrees over all such choices of $k$ and $c$.  For each $q \in f^{-1}(P) \setminus P$, the label $\nu_q$ is defined to be $\nu_p$, where $p=f(q)$.  

So, for example, if $c \in P$ is critical and $f^k(c)=c$ for some $k$ (that is, the portrait for $f$ has a loop based at $c$) then $\nu_c = \infty$. 

It is a fact that every orbifold structure on $(S,P)$ for which $f$ is a partial self-orbifold covering map is a multiple of the one constructed above.  As such, this orbifold structure is often referred to as \emph{the} orbifold for $f$.  

\p{Euler characteristic and hyperbolic orbifolds} The Euler characteristic of an orbifold $(S,P)$ is given by the Riemann--Hurwitz formula
\[
\chi(S,P) = \chi(S) + \sum_P \left(\frac{1}{\nu_p}-1\right)
\]
(here $\chi(S)$ is the usual Euler characteristic for surfaces).  We can think of an orbifold topologically as the surface obtained from $S$ by deleting a disk around each $p \in P$ and gluing in a fraction of a disk, namely, one $\nu_p$th of a disk; hence the formula.  We say that $(S,P)$ is hyperbolic, Euclidean, or spherical if $\chi(S,P)$ is negative, zero, or positive, respectively.

Under an orbifold covering map $f : \Sigma \to T$ of degree $d$ we have the usual multiplicative property 
\[
\chi(\Sigma) = d \cdot \chi(T).
\]
It follows that in an orbifold covering, both orbifolds are of the same type: hyperbolic, Euclidean, or spherical.

\p{Geometric orbifolds} There is an entirely geometric approach to orbifolds.  Let $X$ be $\R^2$, $\H^2$, or $S^2$, and let $G$ be a discrete group of isometries of $X$ (unlike a covering space action, the action of $G$ might not be free).  The quotient $\Sigma = X/G$ is naturally described as an orbifold: the label of a point in $\Sigma$ is the cardinality of the stabilizer in $G$ of any preimage in $X$.  We think of a point labeled $\nu$ as a cone point of order $\nu$.  We refer to any orbifold constructed in this way as a geometric orbifold.  The space $X$ is the orbifold universal cover of $\Sigma$ and $G$ its orbifold fundamental group.   

Thurston determined exactly which orbifolds are geometric \cite[Theorem~13.3.6]{Thurston}.  In particular, he proved that all hyperbolic and Euclidean orbifolds are geometric: they arise as quotients of $\H^2$ and $\R^2$ by discrete groups of isometries as above.  He also proved that all orbifolds with three or more cone points are geometric.  It follows from the Gauss--Bonnet theorem that the space $X \in \{\R^2,\H^2,S^2\}$ is determined uniquely by the orbifold $X/G$.  

\p{Lifting to the universal cover} Now that we have given geometric meaning to the notion of an orbifold, we can do the same for the notion of an orbifold covering map.  Specifically, it is a fact that any orbifold covering map lifts to a map of their orbifold universal covers. In other words, if $f : \Sigma \to T$ is a partial orbifold covering map and $\pi_\Sigma : X \to \Sigma$ and $\pi_T : X \to T$ are the universal covering maps, then there is a map $\tilde f$ so that the following diagram commutes
\[
 \begin{tikzcd}
  X \arrow[r,"\tilde f"] \arrow[d,"\pi_\Sigma",right] & X \arrow[d,"\pi_T"] \\
 \Sigma \arrow[r,"f"]  & T
 \end{tikzcd}
\]
Indeed, the definition of a partial orbifold covering map implies that $f$ induces a well-defined homomorphism of orbifold fundamental groups.  As such, the natural analogue of the usual lifting criterion from algebraic topology applies, implying the existence of $\tilde f$.  If $f$ is holomorphic then, since $\pi_\Sigma$ and $\pi_T$ are holomorphic by definition, the induced map $\tilde f$ is holomorphic.  

\p{Compatible measured foliations} Let $\Sigma = (S,P)$ be a marked surface endowed with a complex structure, and let $(\F^+,\F^-)$ be a pair of transverse measured foliations on $\Sigma$ (as usual, any 1-pronged singularities of the singularities must lie at points of $P$).  Let $Q$ be the set of singular points of the pair of foliations.  The pair $(\F^+,\F^-)$ induces a pair of transverse, nonsingular foliations on $\Sigma \setminus Q$.  Further, these foliations induce a complex structure on $\Sigma \setminus Q$, hence on $\Sigma$ (by the removable singularity theorem).  The charts for this complex structure map open sets in $\Sigma \setminus Q$ to $\C$ in such a way that $(\F^+,\F^-)$ map to the measured foliations on $\C$ given by horizontal and vertical lines, the measures for the latter being $|dy|$ and $|dx|$, respectively.  We say that the pair $(\F^+,\F^-)$ is compatible with the complex structure on $\Sigma$ if the complex structures agree.  

The reader familiar with quadratic differentials will recognize that a compatible pair of foliations on $\Sigma$ is the same as an integrable meromorphic quadratic differential on $\Sigma$ with all (simple) poles at points of $P$.  Since every marked surface with a complex structure admits a nontrivial quadratic differential (on $S_g$ there is a $(6g-6)$-dimensional vector space of these), every complex structure has a compatible pair of measured foliations.  

A pair of measured foliations on $\Sigma$ gives more information than a complex structure: it gives a Euclidean structure on $\Sigma \setminus Q$, and a singular Euclidean structure on $\Sigma$.  In particular, we have an area form as well as a total area.  

\p{The Jenkins extremal problem} Let $\Sigma = (S,P)$ be a marked surface endowed with a complex structure, and let $\Gamma = \{\gamma_1,\dots,\gamma_k\}$ be a labeled multicurve in $\Sigma$.  We denote the weight on $\gamma_i$ by $w(\gamma_i)$.

A multi-annulus in $\Sigma$ is a disjoint union of domains, each biholomorphic to an open annulus in $\C$, and each disjoint from $P$.  We consider the following extremal problem: given the labeled multicurve $\Gamma$ as above, find a multi-annulus $A = \{A_1,\dots,A_k\}$ with the following properties:
\begin{enumerate}
\item each $A_i$ is homotopic to $\gamma_i$,
\item $(\mu(A_1),\dots,\mu(A_k))$ is a multiple of $(w(\gamma_1),\dots,w(\gamma_k))$,
and 
\item $(\mu(A_1),\dots,\mu(A_k))$ is maximal with respect to the first two properties. 
\end{enumerate}
Jenkins proved that when $S$ is not the torus, this extremal problem has a unique solution \cite[Theorem 1]{jenkins}.  This solution corresponds to a pair of measured foliations $(\F^+,\F^-)$ that is compatible with the complex structure.  The singular leaves of $\F^+$ form a finite graph in $S$ (with singular points as vertices) whose complement is a disjoint union of open annuli, each foliated by smooth closed leaves of $\F^-$.  The modulus of each annulus with respect to the complex structure is the modulus of the corresponding Euclidean annulus (the modulus of a Euclidean annulus with circumference $C$ and heights $H$ is $2{\pi}H/C$).  By the uniqueness of the extremal problem, it follows that the pair $(\F^+,\F^-)$ is unique up to scale.  

\p{Strong reduction systems as Thurston obstructions} Let $f : \Sigma \to \Sigma$ be a dynamical branched cover.  Suppose that
\begin{itemize}
\item $f$ is holomorphic and
\item $f$ has a strong reduction system $\Gamma$.
\end{itemize}
We will show that either $\deg f = 1$ or $f$ has Euclidean orbifold.  This means that for $f$ with hyperbolic orbifold and degree greater than 1, strong reduction systems are obstructions to holomorphicity (and vice versa).  

Fix a complex structure on $\Sigma$ with respect to which $f$ is holomorphic (we may have to replace $f$ with a homotopic map).  Let $(A_1,\dots,A_k)$ be the multi-annulus that gives the solution to the Jenkins extremal problem associated to $\Gamma$, and let $(\F^+,\F^-)$ be a corresponding pair of measured foliations.  By the definition of a strong reduction system, the preimage is an equal or larger solution to the extremal problem.  Indeed, the preimage of the collection $(A_1,\dots,A_k)$ is, after consolidating parallel annuli, a multi-annulus where the moduli are given by the weights on $f^*(\Gamma)$ (this uses three basic facts: (1) an $m$-fold cover of annuli multiplies modulus by $m$, (2) the modulus of a union of the closures of two adjacent annuli is the sum of the moduli, and (3) modulus is monotone under inclusion).  

By the previous paragraph, and the fact the compatible foliations for a solution to the Jenkins problem is unique up to scale, it must be that $(f^*\F^+,f^*\F^-)$ is a positive multiple of $(\F^+,\F^-)$.  Moreover, since a cover of degree $d$ reduces Euclidean area by a factor of $d$, we have 
\[
(f^*\F^+,f^*\F^-) = \sqrt{d} \cdot (\F^+,\F^-)
\]
Therefore, if we lift the map $f$ to the universal cover, we obtain a biholomorphic homothety where the scaling factor is $\sqrt{d}$.  Biholomorphic maps of the hyperbolic plane are isometries, and so it must be that $d=1$ or that the orbifold for $f$ is Euclidean, as desired.

\p{Orbifolds and exceptional maps} We have one more loose end to tie up with respect to orbifolds and Thurston's characterization of rational maps.  As promised in the introduction, we explain here why an (unmarked) dynamical branched cover of the sphere has orbifold the $(2,2,2,2)$-orbifold if and only if is a hyperelliptic quotient of a torus map.  This statement is originally due to Cannon--Floyd--Parry--Pilgrim \cite[Theorem 1.4]{CFPP}.

We explained one direction in the introduction: hyperelliptic quotients of torus maps have the $(2,2,2,2)$-orbifold as their orbifold.  Now suppose that $f : (S^2,P) \to (S^2,P)$ is a dynamical branched cover with $(2,2,2,2)$-orbifold.  We would like to show that $f$ lifts---through the hyperelliptic involution---to a map of the torus.  In other words, we would like to show that there is a map $\tilde f$ as in the following diagram:
\[
 \begin{tikzcd}
  T^2 \arrow[r,dashed,"\tilde f"] \arrow[d,"p "] & T^2 \arrow[d,"p "] \\
 (S^2,P) \arrow[r,"f"]  & (S^2,P)
 \end{tikzcd}
\]
where $p$ is the quotient map $T^2 \to T^2/\langle \iota \rangle = (S^2,P)$.
The orbifold fundamental group of $(S^2,P)$ has the presentation
\[
\pi_1^{orb}(S^2,P) \cong \langle a_1,a_2,a_3,a_4 \mid a_1^2=a_2^2=a_3^2=a_4^2=abcd=1 \rangle
\]
and the image of the induced map
\[
p_* : \pi_1(T^2) \to \pi_1^{orb}(S^2,P)
\]
is the even subgroup of $\pi_1^{orb}(S^2,P)$, that is, the kernel of the map
\begin{align*}
\pi_1^{orb}(S^2,P) &\to \Z/2 \\
a_i &\mapsto 1.
\end{align*}
Since all four points of $P$ carry the label 2, it follows that the local degree of $f$ at each point of $P$ is 1.  Thus, 
the induced map $f_*$ maps the even subgroup of $\pi_1^{orb}(S^2,P)$ to itself.  Finally, by the lifting criterion for orbifold covering maps implies the existence of $\tilde f$, as desired.


\section{Topological polynomials, Levy cycles, and Levy--Berstein} 
\label{sec:levy}

In this appendix we prove a strong form of the theorem which says that if a topological polynomial is not rational then it has a degenerate Levy cycle.  Again, this theorem is due to the work of Berstein, Hubbard, Levy, Rees, Tan, and Shishikura.  Our strengthening is Proposition~\ref{prop:levy} below.  In the statement, we say that a strong reduction system is \emph{minimal} if all multicurves with fewer components fail to underlie a strong reduction system.  If a dynamical branched cover has a strong reduction system, then it has a minimal one.  

\begin{proposition}
\label{prop:levy}
Let $f \colon (\R^2,P) \to (\R^2,P)$ be a topological polynomial.  Every minimal strong reduction system for $f$ is a degenerate Levy cycle.  In particular, if $f$ has a strong reduction system then it has a degenerate Levy cycle.  
\end{proposition}

As in the introduction, the Levy--Berstein theorem says that if $f$ is a topological polynomial and each point of $P$ has a critical point in its forward $f$-orbit then $f$ is rational.  This is immediate from Proposition~\ref{prop:levy}, since the union of the disks for a degenerate Levy cycle contains no critical points.  

Our argument for Proposition~\ref{prop:levy} is a modification of the argument in Hubbard's book for an analogous statement about Thurston obstructions \cite[Theorem 10.3.7]{hubbard}.  We use two tools, innermost curves and lifting graphs.  

\p{Innermost curves} The main feature that makes topological polynomials different from topological rational maps---and what allows us to prove Proposition~\ref{prop:levy}---is that every curve in $(\R^2,P)$ has a well-defined interior: the compact region of $\R^2$ bounded by the curve.  Moreover, if $\delta$ is a component of $f^{-1}(\gamma)$ then $f$ maps the interior of $\delta$ onto the interior of $\gamma$.  Given a multicurve $\Gamma$ we will denote by $\Gamma^\circ$ the multicurve given by its innermost components.  

\p{Lifting graphs} If $\Gamma$ is an $f$-stable, labeled multicurve for a dynamical rational map $f$, we define a corresponding a directed graph, the lifting graph, as follows: the vertices are the components of $\Gamma$ and there is a directed edge from $\gamma$ to $\delta$ if $\delta$ is homotopic to a component of $f^{-1}(\gamma)$ (note that $f^{-1}(\gamma)$ may have components that are inessential or are essential and not homotopic to a component of $\Gamma$).  We label each vertex by the corresponding labels on the curves of $\Gamma$ and we label each edge by a natural number, the degree of $f|\delta : \delta \to \gamma$.

We can interpret the action of $f^*$ on $\Gamma$ in terms of the lifting graph.  Under $f^*$, the labels on the vertices change as follows: the new label on a vertex $v$ is the sum of $w_i/d_i$ where $w_i$ is the weight on the $i$th vertex with a directed edge pointing to $v$ and $d_i$ is the label on that edge.  

\begin{proof}[Proof of Proposition~\ref{prop:levy}]

Let $\Gamma$ be a labeled multicurve in $(\R^2,P)$ giving a minimal strong reduction system for $f$.  Let $G$ be the corresponding lifting graph.  

We first claim that each vertex of $G$ has at least one incoming edge, that is, $G$ has no initial vertices.  This follows from the stability of $\Gamma$, since an initial vertex would be a component of $\Gamma$ not parallel to a component of $f^{-1}(\Gamma)$.  

We next claim that each vertex of $G$ has at least one outgoing edge, that is, $G$ has no terminal vertices.  Indeed, suppose that a vertex $\gamma$ is terminal.  It cannot be that $\gamma$ is the only vertex of $G$, for then $G$ would have no edges, and it would be impossible for $\Gamma = \gamma$ to underly a strong reduction system.  Now, if we delete $\gamma$ from $\Gamma$, then the multicurve that remains---which is nonempty by the previous sentence---still underlies a strong reduction system for $f$, violating the minimality of $\Gamma$.  

We now claim that the set of vertices of $G$ corresponding to innermost curves of $\Gamma$ determines a closed subgraph $G^\circ$ of $G$, that is, a directed edge starting at an innermost curve ends at an innermost curve.  Suppose there is a directed edge from some curve $\gamma$ to a curve $\delta$ that is not innermost.  We will show that $\gamma$ is not innermost.  Let $\epsilon$ be a curve of $\Gamma$ in the interior of $\delta$ (and not parallel to $\delta$).  Since $G$ has no initial vertices, $\epsilon$ lies in the $f$-preimage of a curve $\phi$ of $\Gamma$.   And because $f$ maps interiors to interiors, this $\phi$ would have to lie in the interior of $\gamma$.  Also, since the components of $\Gamma$ are not parallel pairwise and since $f$ is a function, $\phi$ is not parallel to $\gamma$, and the claim is proved.

We next claim that $G^\circ$ is equal to $G$.  Suppose not.  Then the subgraph $G'$ of $G$ spanned by the vertices not in $G^\circ$ is nonempty.  We will show that $G'$ represents a strong reduction system for $f$, which will violate the minimality of $\Gamma$.  We first show that $G'$ represents a stable multicurve, and then check the condition on labels.  Since $G$ has no terminal vertices, each vertex of $G'$ is the end point of an edge of $G$.  As $G^\circ$ is closed, it must be that the edges terminating in $G'$ have origins in $G'$.  This is to say that $G'$ represents a stable multicurve for $f$.  The action of $f^*$ on the labels of $G'$ agrees with the restriction of its action on the labels of $G$, and so $G'$ does indeed represent a strong reduction system for $f$, the desired contradiction.  

We now claim that no two directed edges of $G$ have the same endpoint.  Indeed, by the previous claim all vertices of $G$ are innermost curves of $\Gamma$.  Any two innermost curves are un-nested, that is, neither lies in the interior of the other.  It follows that the components of their preimages un-nested.  In particular, the preimages cannot be parallel, whence the claim.

At this point, we have shown that $G$ has no initial or terminal vertices and that no two edges has the same endpoints.  It follows that $G$ is a union of directed cycles.  By minimality, $G$ is a single directed cycle.

If $G$ has an edge label greater than 1, then there are no positive labels of the vertices of $G$ that satisfy the condition for a strong reduction system.  Therefore all of the edges are labeled 1.  This is to say that $G$ represents a Levy cycle.  Since each curve of $\Gamma$ maps to the next with degree 1, the disks interior to these curves also map to the next with degree 1, meaning that $\Gamma$ is a degenerate Levy cycle, as desired.
\end{proof}


\section{Further extensions of the \"Ubertheorem}
\label{sec:ext}

In this third and final appendix, we explain several generalizations of the Nielsen--Thurston \"Ubertheorem.  There are three versions: for equivariant maps, for non-orientable surfaces, and for orientation-reversing maps.  All of these are straightforward extensions of the \"Ubertheorem.  In theory, we could combine all of the extensions into one Super\"ubertheorem, but for clarity we prefer to state them separately.  We also state the extensions informally, because some of the details are left to the reader.

\p{Equivariant maps} Let $\Sigma = (S,P)$, let $f : \Sigma \to \Sigma$ be a dynamical branched cover.  Let $G$ be a finite group that acts on $\Sigma$.  As usual, we say that $f$ is $G$-equivariant is $f(g \cdot x) = g \cdot f(x)$ for all $x \in S$.  For example, we say that $f$ is an odd map of $(S^2,P)$ if it is $\Z/2$-equivariant, where $\Z/2$ acts by the antipodal map.

If we assume that the map $f$ in the statement of the \"Ubertheorem is $G$-equivariant, then the \"Ubertheorem (of course) still holds, but with the added conclusion that the resulting homotopic map $\phi$ is also $G$-equivariant.  We have the following consequences:
\begin{enumerate}
\item if $\phi$ is holomorphic then $G$ preserves the complex structure,
\item if $\phi$ is strongly reducible, then $G$ preserves the strong reduction system, and
\item if $\phi$ is pseudo-Anosov, then $G$ preserves the measured foliations.
\end{enumerate}
The key observation required to prove this enhancement of the \"Ubertheorem is that the pullback of any geometric object (complex structure, strong reduction system, measured foliation, etc.) under a $G$-equivariant map is $G$-invariant.  So, for example, the image of the pullback map $\sigma_f$ is contained in the subspace of $\Teich(\Sigma)$ fixed by the action of $G$.

\p{Non-orientable surfaces} For a non-orientable, closed surface $S$, we can define a marked surface $\Sigma = (S,P)$ and a dynamical branched cover $f : \Sigma \to \Sigma$ as in the orientable case.  Such maps arise naturally even when studying dynamical branched covers of orientable surfaces.  For instance, any odd map of $\Sigma = (S^2,P)$ descends to a dynamical branched cover of $(\rpt,\bar P)$ where $\bar P$ is the image of $P$ under the quotient of $S^2$ by the antipodal map.

The natural analogue of a complex structure in this setting is a conformal structure, by which we mean a map that preserves angles, up to sign, in the tangent space.  This is equivalent to the existence of an atlas where the charts map to the complex plane and transition maps are holomorphic or anti-holomorphic.  For orientable surfaces, complex structures and conformal structures are the same thing.  

Given $f : \Sigma \to \Sigma$ for non-orientable $\Sigma$, we obtain a dynamical branched cover $\tilde f : \tilde \Sigma \to \tilde \Sigma$ of the orientation double cover $\tilde \Sigma$.  The deck group for this (characteristic) cover is $G \cong \Z/2$ and the map $\tilde f$ is $G$-equivariant.  As above the map $\tilde f$ is (up to homotopy) either holomorphic, strongly reducible, or pseudo-Anosov.  And moreover these corresponding geometric structures are $G$-invariant.  These means that $f$ is either conformal, strongly reducible, or pseudo-Anosov, giving our second extension of the \"Ubertheorem.

There is an important subtlety in the above argument.  When we modify $\tilde f$ by isotopy, we need to know that we can modify $f$ accordingly.  In other words, we need to know that the isotopy of $\tilde f$ can be pushed down to an isotopy of $f$.  

In the theory of mapping class groups, it is true that homotopic $G$-equivariant homeomorphisms are $G$-equivariantly homotopic; this fact is known as the Birman--Hilden theorem (see the expository paper by the second- and third-named authors \cite{MW}).  The analogue of the Birman--Hilden theorem does indeed hold for $G$-equivariant maps of degree greater than 1 (which is what we need here).  In fact, the Maclachlan--Harvey proof of the Birman--Hilden theorem, which is based on Teichm\"uller theory, applies almost directly to this more general case (see page 13 of the aformentioned survey for a discussion).  The only change needed is to replace all of the groups in the proof with monoids, since maps of degree greater than 1 do not have inverses.

\p{Orientation-reversing maps}  Let $\Sigma = (S,P)$ be a marked surface, and suppose that $\Sigma$ is oriented.  We say that an orientation-reversing map $f : \Sigma \to \Sigma$ is a dynamical branched cover if $f$ restricts to an (unbranched) covering space over $S \setminus P$.  One way to construct such an $f$ is to take an (orientation-preserving) dynamical branched cover $(S^2,P) \to (S^2,P)$ where $P$ is preserved by the antipodal map and post-compose with the antipodal map.   

Let $f : \Sigma \to \Sigma$ is an orientation-reversing dynamical branched cover.  We claim that $f$ is homotopic to a map that is either anti-holomorphic, strongly reducible, or pseudo-Anosov, and moreover this follows from our proof of the \"Ubertheorem.  The only required observation is that if an orientation-reversing map fixes a point in Teichm\"uller space then it is anti-holomorphic with respect to the corresponding complex structure.  For the non-exceptional cases, this statement was already stated and proved by Geyer \cite[Theorem 3.9]{geyer}.


\bibliographystyle{plain}
\bibliography{thurston}

\end{document}